\documentclass[12pt]{article}
\usepackage[margin=1in]{geometry}
\usepackage[utf8]{inputenc}

\usepackage[english]{babel}

\usepackage[round]{natbib}
\usepackage{xcolor}

\usepackage{mathtools,amssymb,amsthm}
\usepackage{graphicx,color}
\usepackage{subcaption}
\usepackage{pifont}
\usepackage{booktabs}
\usepackage{microtype}
\usepackage{dsfont}

\usepackage{float}
\usepackage[boxed]{algorithm2e}

\usepackage[toc,page]{appendix}
\usepackage{hyperref}
\usepackage[capitalize]{cleveref}
\usepackage{bm}
\usepackage{enumitem}

\newcommand{\cN}{\mathcal{N}}

\newcommand{\cP}{\mathcal{P}}

\newcommand{\cT}{\mathcal{T}}

\newcommand{\PP}{\mathbb{P}}

\newcommand{\RR}{\mathbb{R}}

\newcommand{\R}{\RR}
\newcommand{\N}{\mathbb{N}}
\newcommand{\E}{\mathbb{E}}

\newcommand*{\kl}[3][]{%
\ifthenelse{\isempty{#1}}{\operatorname{D}(#2\,\|\,#3)}%
{\operatorname{D}(#2\,\|\,#3\mid#1)}%
}

\newcommand*{\triplenorm}[1]{{\left\vert\kern-0.25ex\left\vert\kern-0.25ex\left\vert #1
    \right\vert\kern-0.25ex\right\vert\kern-0.25ex\right\vert}}

\DeclareMathOperator{\Cov}{Cov}

\newcommand*{\defeq}{\coloneqq}
\newcommand*{\rd}{\mathrm{d}}
\newcommand*{\dd}{\, \rd}

\newcommand{\eps}{\varepsilon}

\renewcommand{\phi}{\varphi}

\newcommand{\mmid}{\,\Vert\,}
\newcommand{\deq}{\coloneqq}

\DeclareMathOperator*{\KL}{\mathsf{H}}
\DeclareMathOperator*{\FI}{\mathsf{FI}}

\newcommand*{\norm}[1]{\left\|#1\right\|}
\newcommand{\comp}{\mathsf c}

\renewcommand{\geq}{\geqslant}
\renewcommand{\ge}{\geqslant}
\renewcommand{\leq}{\leqslant}
\renewcommand{\le}{\leqslant}

\newcommand{\KM}{\mathsf T}
\newcommand{\mom}{\mathsf m_2}
\DeclareMathOperator{\id}{id}
\usepackage{eufrak}

\hypersetup{citecolor=purple, urlcolor=cyan, colorlinks=true, linkcolor=blue}

\theoremstyle{plain}
\newtheorem{theorem}{Theorem}[section]
\newtheorem{proposition}[theorem]{Proposition}

\newtheorem{lemma}[theorem]{Lemma}

\theoremstyle{definition}
\newtheorem{definition}[theorem]{Definition}

\theoremstyle{remark}
\newtheorem{remark}[theorem]{Remark}

\title{Near-Lipschitz stability of the 
Kim--Milman flow map}

\author{
Sinho Chewi\thanks{Department of Statistics and Data Science, Yale University. \tt{sinho.chewi@yale.edu}}\and 
Katharina Eichinger\thanks{ParMA, Inria Saclay and Laboratoire de mathématiques d’Orsay, Université Paris-Saclay. \tt katharina.eichinger@inria.fr}\and
Aram-Alexandre Pooladian\thanks{Institute for Foundations of Data Science, Yale University. \tt aram-alexandre.pooladian@yale.edu} 
}

\date{\today}

\begin{document}
\maketitle
\begin{abstract}
We prove that the Kim--Milman flow map enjoys favorable stability properties with respect to variations in the target measure, provided that one of the {target} measures is sufficiently regular. Our results include stability in relative entropy, and more notably, Lipschitz stability in the $2$-Wasserstein distance  up to a logarithmic factor. We complement our results with a general existence theorem for these maps for any target measure with finite second moment.
\end{abstract}

\section{Introduction}
Let $\mu$ be a probability measure over $\R^d$ and let $\gamma \deq \cN(0,I)$. The seminal work of~\cite{kim2012generalization} introduced a flow map $\mathsf T^{\mu}$ which transports $\gamma$ to $\mu$, i.e., for $X \sim \gamma$, $\mathsf T^{\mu}(X)\sim\mu$. Unlike other transport maps, such as optimal transport maps \citep{brenier1991polar} or the Schr\"odinger bridge \citep{Leo14Schrodinger}, the study of $\mathsf T^{\mu}$ does not typically proceed by viewing it as the solution to a variational problem
but instead leverages the theory of forward and reverse diffusions \citep{Fol85}; see Section~\ref{sec:construction} for more details on its construction. We henceforth call $\mathsf T^\mu$ the Kim--Milman flow map or simply Kim--Milman map.

The Kim--Milman map was developed as a means of generalizing Caffarelli's celebrated contraction theorem beyond the case in which $\mu$ is strongly log-concave \citep{caffarelli2000monotonicity}, and has remained a subject of intense investigation in recent years; see for example \cite{neeman2022lipschitz, klartag2023spectral, MikShe23HeatFlow, fathi2024transportation, BriPed25HeatFlow,conforti2025coupling, ge2025generalization, lopez2025bakry, serres2026contractive} for developments and applications in pure mathematics.

Whereas the works above mostly focus on regularity of $\mathsf T^\mu$ as a function of its input, here we focus on its \emph{stability} with respect to the targets: For two measures $\mu$ and $\nu$, can we accurately quantify the distance between the maps $\mathsf T^\mu$ and $\mathsf T^\nu$
by some (pseudo-)metric between the measures? This question is motivated, in part, by extensive work on the analogous question for optimal and entropic transport maps, which we review in Section~\ref{sec:related}.

Our main result is the following theorem, which states that if a Kim--Milman flow map is Lipschitz (with respect to its input), then it is also nearly Lipschitz stable with respect to variations in the target measure in the $2$-Wasserstein distance.
\begin{theorem}[Informal; see Theorem~\ref{thm:wass_stab}]\label{thm:intro}
Suppose that $\mu$ satisfies a suitable regularity condition, implying in particular that $\mathsf T^\mu$ is Lipschitz, and that $\nu \in \cP(\mathbb{R}^d)$ has a finite second moment. Then,
\begin{align*}
    \|\mathsf T^\mu - \mathsf T^\nu\|_{L^2(\gamma)} \le c(\mu)\, W_2(\mu,\nu) \log\Bigl(2 + \frac{c'(\mu)}{W_2(\mu,\nu)}\Bigr)\,, 
\end{align*}
where $c(\mu)$, $c'(\mu)$ are constants that depend only on $\mu$.
\end{theorem}
The regularity condition assumed in Theorem~\ref{thm:intro} is ubiquitous in the literature, because it implies the Lipschitzness of the map $\KM^\mu$; we formally state the assumption later as the condition \eqref{theta}.
Thus, our theorem shows that this same condition suffices for quantitative stability bounds as well.

We mention here that this result is perhaps surprising, since optimal transport maps are at most $\frac{1}{3}$-H\"older stable {w.r.t. the $2$-Wasserstein distance} in general (see Section~\ref{sec:related}). 

Our proof uses the flow map interpretation of the Kim{--}Milman map and reduces the stability bound to estimates for the relative Fisher information along the Ornstein--Uhlenbeck process. For large times, we combine the de Bruijn identity with a log-Harnack inequality to bound this in terms of the initial Wasserstein distance between the measures. For short times, we use an alternative bound which has better time dependence, but which does not depend on the initial Wasserstein distance. Balancing these two regimes yields our result. We believe the logarithmic factor is an artifact of our proof technique, although Section~\ref{sec:lower_bound} discusses how certain estimates in our proof cannot be improved. 

As a final contribution, we prove that Kim--Milman flow maps exist for any probability measure with finite second moment; see Lemma~\ref{lem:existence_km}. This result follows from a purely technical argument that exploits the regularizing property of the relative Fisher information.

\subsubsection*{Acknowledgments}
KE acknowledges the support of the Agence nationale de la recherche, through the PEPR PDE-AI project (ANR-23-PEIA-0004). AAP thanks the Institute for Foundations of Data Science at Yale University for financial support. 

\section{Notation, preliminaries, and related work}
We let $\cP_2(\R^d)$ denote the space of probability measures on $\R^d$ with finite second moment, and let $\cP_{2,\rm ac}(\R^d)$ be the subset of absolutely continuous measures. For two measures $P,Q \in \cP_2(\R^d)$, we write the relative entropy between them as 
\begin{align*}
{\mathsf H}(P\mmid Q) \defeq \int \log \bigl(\frac{\dd P}{\dd Q}\bigr) \dd P\,,
\end{align*}
whenever $P \ll Q$, with  ${\mathsf H}(P\mmid Q) = +\infty$ otherwise. We similarly define the relative Fisher information between $P$ and $Q$ as
\begin{align*}
\FI(P\mmid Q) \defeq \E_P\|\nabla \log P - \nabla \log Q\|^2\,,
\end{align*}
whenever this expression makes sense; in this paper, we will only ever consider the relative Fisher information between $C^\infty$ densities, due to the regularization along the Ornstein--Uhlenbeck process.

For any $P,Q \in \cP_2(\R^d)$, recall that the $2$-Wasserstein distance is
\begin{align*}
W_2(P,Q) \defeq \Bigl(\inf_{\pi \in \Pi(P,Q)} \iint \|x-y\|^2\, \pi(\!\dd x,\!\dd y)\Bigr)^{1/2}\,,
\end{align*} 
where $\Pi(P,Q)$ is the space of joint probability measures with marginals $P$ and $Q$. For any $P \in \cP_2(\R^d)$, we write $\mom(P) \defeq (\int \|x\|^2\dd P(x))^{1/2} < \infty$.  

\subsection{Construction of the map}\label{sec:construction}

We first provide some background on the construction of the flow map due to \cite{kim2012generalization}, which is based on time reversal of the Langevin dynamics. We differ from the usual construction in the literature based on inverting the forward map, and instead first invert the time before taking the limiting flow map.

Let $\gamma = \cN(0,I)$ and let $\mu \in \cP_2(\R^d)$ be arbitrary. Our forward dynamics are governed by the Ornstein--Uhlenbeck (OU) process: for $X_0^{\rightarrow} \sim \mu$, 
\begin{align}\label{eq:ou}
    \dd X_t^\rightarrow = - X_t^{\rightarrow}{\dd t} + \sqrt{2}\dd B_t\,,
\end{align}
where $(B_t)_{t \geq 0}$ is a standard Brownian motion. Writing $q_t^\mu \defeq {\rm Law}(X_t^{\rightarrow})$, $(q_t^\mu)_{t\ge 0}$ satisfies the following Fokker--Planck equation\footnote{Here we  abuse notation and identify measures with their Lebesgue densities.}
\begin{align*}
    \partial_t q_t^\mu - \nabla \cdot (q_t^\mu\, ({\rm id} + \nabla \log q_t^\mu)) = 0\,, \quad q_0^\mu = \mu\,.
\end{align*}
Observe that if $\mu \in \cP_2(\mathbb{R}^d)$, a simple coupling argument yields that $q_t^\mu \to \gamma$ (for instance, in $W_2$) as $t \to \infty$. 

Performing a time reversal on the interval $[0,T]$ of this equation yields
\begin{align*}
    \partial_t q^\mu_{T-t} + \nabla \cdot (q^\mu_{T-t}\,({\rm id} +  \nabla \log q^\mu_{T-t})) = 0\,. 
\end{align*}
Recognizing this as a continuity equation, it describes the evolution of the marginal law for the following ODE, i.e., ${\rm Law}(X_t^{\leftarrow}) = q^\mu_{T-t}$ with
\begin{align}\label{eq:rev_ou_ode}
    \dot X_t^{\leftarrow} = X_t^{\leftarrow} + \nabla \log q^\mu_{T-t}(X_t^\leftarrow) = \nabla \log\Bigl( \frac{\!\dd q^\mu_{T-t}}{\!\dd \gamma}\Bigr)(X_t^\leftarrow) \,, \quad X_0^\leftarrow \sim q^\mu_T\,.
\end{align}
We can also write \eqref{eq:rev_ou_ode} using semigroup notation. Letting $(Q_s)_{s \geq 0}$ denote the OU semigroup, using $\frac{\dd q^\mu_{T-t}}{\dd \gamma} = Q_{T-t}\frac{\dd \mu}{\dd \gamma}$ we have that
\begin{align}\label{eq:ode_semigroup}
    \dot X_t^{\leftarrow} 
    = \nabla \log Q_{T-t}\Bigl[\frac{\!\dd\mu}{\!\dd\gamma}\Bigr](X_t^{\leftarrow})\,, \quad X_0^\leftarrow \sim q^\mu_T\,,
\end{align}
whenever $\mu \ll \gamma$.
Here, we made use of the fact that $\gamma$ is the stationary distribution of the OU process.

Now, let $\mathsf T^{\mu,T} \equiv (\mathsf T^{\mu,T}_t)_{t \in [0,T]}$ denote the flow map solution to \eqref{eq:rev_ou_ode}. So, $\mathsf T^{\mu,T}_{t}(X^{\leftarrow}_0) = X^{\leftarrow}_t$ and $(\mathsf T^{\mu,T}_{T-t})_\sharp q^\mu_T = q^\mu_t$. We now take the limit $T \to\infty$, and the transport map from $\gamma$ to $\mu$ is thus given by $\mathsf T^{\mu,\infty} \defeq \lim_{T\to\infty} \mathsf T^{\mu,T}$.
{We also abbreviate $\mathsf T^\mu \deq \mathsf T^{\mu,\infty}$.}

We will also make use of the flow map from $\gamma$ to $q^\mu_t$ given by
\begin{equation}\label{eq:flow-map-q_t}
    \mathsf T^{\mu}_{t \gets} \deq \lim_{T \rightarrow \infty} \mathsf T^{\mu,T}_{T-t}
\end{equation}
which satisfies $(\KM^{\mu}_{t\gets})_{\sharp}\gamma = q^\mu_t$.
The notation indicates that this is the KM map, stopped at $q_t^\mu$.

\subsection{Assumption and existing properties}\label{sec:existing}

\subsubsection{Tweedie identities}

For later reference, we collect together the basic Tweedie identities here.

\begin{lemma}
    Let $X_0 \sim \mu$ and $Z \sim \gamma$ be independent, and write $X_t = e^{-t}\,X_0 + \sqrt{1-e^{-2t}}\,Z$.
    Then, for all $t > 0$,
    \begin{align}
        \nabla \log \frac{q_t^\mu}{\gamma}(x_t)
        &= -\frac{e^{-t}}{1-e^{-2t}}\, \E[e^{-t} X_t - X_0 \mid X_t = x_t]\,,\label{eq:tweedie_one} \\[0.25em]
        \nabla^2 \log \frac{q_t^{\mu}}{\gamma}(x_t)
        &= -\frac{e^{-2t}}{1-e^{-2t}}\,I + \frac{e^{-2t}}{(1-e^{-2t})^2} \Cov(X_0 \mid X_t=x_t)\,.\label{eq:tweedie_two}
    \end{align}
\end{lemma}

\subsubsection{Existence and Lipschitz continuity}

The (reverse) heat flow map was developed by \citet{kim2012generalization} as a means of providing a generalization of Caffarelli's contraction theorem.
Since then, there has been much progress towards proving Lipschitz continuity of the heat flow map; for recent developments see \cite{MikShe23HeatFlow}, \cite{fathi2024transportation, conforti2025coupling} for general Langevin dynamics, and \cite{BriPed25HeatFlow, chaintron_propagation_2025} for the most general treatment in the case of the OU process.

All of the above works have in common finding sufficient conditions for the following  pointwise log-semi-concavity bound and accompanying integrability condition:
\begin{align}\label{theta}\tag{$\Theta$}
    \nabla^2 \log Q_s\Bigl[\frac{\!\dd\mu}{\!\dd\gamma}\Bigr](\cdot) = I + \nabla^2 \log q^\mu_s(\cdot) \preceq \theta_s I \quad \text{and}\quad \sup_{t\ge 0}\int_t^\infty \theta_s \dd s < +\infty\,.
\end{align}
This implies Lipschitz continuity of the heat flow map, as stated in the proposition below, and corresponds to our main assumption.
Note also that by~\eqref{eq:tweedie_two}, we always have the lower bound
\begin{align*}
    \nabla^2 \log Q_s\Bigl[\frac{\!\dd\mu}{\!\dd\gamma}\Bigr](\cdot) \succeq -\frac{e^{-2s}}{1-e^{-2s}}\,I\,.
\end{align*}

Lipschitz continuity of the flow map follows from Proposition~\ref{prop:lip_km} below, first given by \cite{kim2012generalization} and later formalized in this form by \citet[Lemma 11]{MikShe23HeatFlow}. We write the statement for the general flow map from \eqref{eq:flow-map-q_t}.

\begin{proposition}\label{prop:lip_km}
Assume that $\mu$ satisfies \eqref{theta}. Then, for all $t \ge 0$, $\mathsf T^{\mu}_{t\gets}$ exists and is Lipschitz continuous with constant at most 
\begin{align}\label{eq:def-L_t}
    L_t(\mu) \deq \exp(\textstyle \int_t^\infty \theta_s \dd s)\,.
\end{align}
\end{proposition}

We note that~\citet{MikShe23HeatFlow} proved the existence of the map {under the condition \eqref{theta}, but not that it admits an ODE representation. Since we need this fact for our results, we prove this as Lemma~\ref{lem:km_regular_case} in Appendix~\ref{app:existence} (see also Section~\ref{sec:ode_rep}).}

Note that the Lipschitz continuity with constant $L_t(\mu)$ implies in particular that $q^\mu_t$ satisfies a \emph{log-Sobolev inequality}: for any $\rho \in \cP(\R^d)$ such that $\rho \ll q_t^\mu$,
\begin{align}\label{eq:lsi}
    \mathsf H(\rho\mmid q_t^\mu) \leq \frac{L_t^2(\mu)}{2} \FI(\rho\mmid q_t^\mu)\,.
\end{align}
As an example, suppose that $\mu$ is $\alpha$-strongly log-concave. Then, it is known that \eqref{theta} is satisfied with~\citep[see, e.g.,][]{MikShe23HeatFlow,BriPed25HeatFlow}
\begin{align}\label{eq:theta_slc}
    I + \nabla^2 \log q_s^\mu(\cdot) \preceq \frac{1-\alpha}{\alpha(e^{2s}-1) + 1}\,I\,.
\end{align}
One can verify that for all $\alpha > 0$, $L_0(\mu) = \alpha^{-1/2}$.

In our main stability result, we will assume that $\mu$ satisfies~\eqref{theta}. However, we do not wish to impose the assumption~\eqref{theta} on $\nu$ as well, but then the existence of the KM map for $\nu$ is not guaranteed by Proposition~\ref{prop:lip_km}. To address this gap, we state a general existence result for any measure $\nu \in \cP_2(\R^d)$ below. We were unable to find this result in the literature, so it could be of independent interest. We defer the proof to Appendix~\ref{app:existence}.

\begin{lemma}[Existence]\label{lem:existence_km}
    Let $\nu \in \cP_2(\R^d)$, and let $\{\nu_n\}_{n\in\N}$ be a regular approximating sequence for $\nu$ (see Definition~\ref{def:reg_approx}).
    Then, for every $t > 0$, the $L^2(\gamma)$ limit $\KM^\nu_{t\gets} \deq \lim_{n\to\infty} \KM^{\nu_n}_{t\gets}$ exists, is independent of the approximating sequence $\{\nu_n\}_{n\in\N}$, and $(\KM^\nu_{t\gets})_\sharp \gamma = q_t^\nu$.

    Moreover, $\KM^\nu \deq \lim_{t\searrow 0} \KM^\nu_{t\gets}$ exists in $L^2(\gamma)$, and $(\KM^\nu)_\sharp \gamma = \nu$.
\end{lemma}

\subsubsection{Stability in relative Fisher information}

Recently, \cite{chewi2025stability} proved the first stability bound for the Kim--Milman flow map. If $\mu$ is such that \eqref{theta} holds, they proved that 
\begin{align}\label{eq:stab_fisher_old}
    \|\mathsf T^{\mu} - \mathsf T^{\nu}\|_{L^2(\gamma)}\lesssim \sqrt{\FI(\nu\mmid\mu)}\,,
\end{align}
where the underlying constant depends on \eqref{theta}. Through similar arguments, they also establish stability of the flow map under stronger metrics:
\begin{align*}
    \|\mathsf T^{\mu} - \mathsf T^{\nu}\|_{L^\infty(\gamma)}\lesssim \sqrt{{\mathsf{FI}}_\infty(\nu\mmid\mu)}\,,
\end{align*}
where $\FI_\infty(\nu\mmid\mu) = {\rm ess} \sup_\nu\|\nabla \log (\!\dd\nu/\!\dd\mu)\|^2$.

The proof of our stability result builds upon the proof therein, supplementing it with new ingredients in order to obtain bounds involving only $W_2(\mu,\nu)$ on the right-hand side. Moreover, in Theorem~\ref{thm:stab_ent}, we prove a stability result in which the right-hand side of~\eqref{eq:stab_fisher_old} is replaced by $\sqrt{\KL(\nu\mmid \mu)}$, which strictly improves upon~\eqref{eq:stab_fisher_old}.

We mention that a similar result also appears in \cite{khudiakova2025optimal}, though with a different motivation, proof, and result.

\subsection{Related work}\label{sec:related}
Inspired by the work of \cite{caffarelli2000monotonicity}, \cite{kim2012generalization} showed that if $\mu$ is $\alpha$-strongly log-concave, then it holds that both the optimal transport map and the KM map are $\alpha^{-1/2}$-Lipschitz \citep{caffarelli2000monotonicity,kim2012generalization}. Recent techniques have extended this result to new scenarios. For example, new results based on entropic optimal transport have extended Caffarelli's original result for optimal transport maps \citep{FatGozPro20Caff, chewi2023entropic,gozlan2025global}. Similarly, the regularity estimates for $\mathsf T^{\mu,\infty}$ have recently been generalized by \cite{MikShe23HeatFlow,BriPed25HeatFlow,chaintron_propagation_2025}, going well beyond the strongly log-concave setting originally considered by \cite{kim2012generalization}.

\paragraph{Stability of optimal transport maps.}
Letting $\cT(\rho,\mu)$ be the set of transport maps from $\rho$ to $\mu$, optimal transport maps are those which minimize the average cost of displacement:
\begin{align}
    \mathsf  T_{0}^{\rho\to\mu} \defeq \underset{\mathsf  T \in \cT(\rho,\mu)}{\rm argmin}\int \|x- \mathsf T(x)\|^2 \, \rho(\!\dd x)\,.
\end{align}
Remarkably, these maps are gradients of convex functions \citep{brenier1991polar}. The stability of these transport maps has been widely studied across various communities \citep{gigli2011holder,delalande2023quantitative, balakrishnan2025stability, letrouit2025quantitative,cazelles2026statistical, merigot2026sharp}.

This line of work was inaugurated by \citet{gigli2011holder}, who proved that if $D\mathsf  T_0^{\rho\to\mu} \preceq L I$ for some constant $L > 0$, then for $\rho,\mu,\nu \in \cP(B(0,R))$ the following stability bound holds:
\begin{align*}
    \|\mathsf  T_0^{\rho\to\mu} - \mathsf  T_0^{\rho\to\nu}\|_{L^2(\rho)}\leq (4RL)^{1/2}\, W_2^{1/2}(\mu,\nu)\,. 
\end{align*}
\cite{manole2024plugin} proved that for any triple of probability measures $\rho,\mu,\nu$, so long as  $0 \prec \ell I \preceq D \mathsf T_0^{\rho\to\mu} \preceq L I$ for $\ell,L> 0$, the stability bound for optimal transport maps can be augmented to
\begin{align}\label{eq:ot_lip_stab}
    \|\mathsf  T_0^{\rho\to\mu} - \mathsf  T_0^{\rho\to\nu}\|_{L^2(\rho)}\leq \Bigl(\frac{L}{\ell}\Bigr)^{1/2}\, W_2(\mu,\nu)\,.
\end{align}
Recent work has sought to characterize stability properties of optimal transport maps in the absence of regularity assumptions. Indeed, \cite{gigli2011holder} provided a construction in which
\begin{align*}
    \|\mathsf T_0^{\rho\to\mu} - \mathsf  T_0^{\rho\to\nu}\|_{L^2(\rho)}\gtrsim W_2^{1/2}(\mu,\nu)\,.
\end{align*}
{Recently, \cite{letrouit_unstable_2026} has improved upon this lower bound by providing a counterexample for Hölder exponent $\frac{1}{3}$.}
Thus, in general, the best stability bound possible is a $\frac{1}{3}$-H\"older bound.
To this end, \cite{MerDelCha20Stability,delalande2023quantitative,letrouit2024gluing,cazelles2026statistical,merigot2026sharp} made recent progress in closing this gap; see \citet{kitagawa2025stability} for the Riemannian setting. The most recent result is by 
\cite{merigot2026sharp}, who proves that for $\rho$ with density bounded from above, and $\mu,\nu \in \cP(B(0,R))$, 
\begin{align}\label{eq:one_over_four}
    \|\mathsf T_0^{\rho\to\mu} - \mathsf T_0^{\rho\to\nu}\|_{L^2(\rho)}\leq C(d,R)\, W_1^{1/4}(\mu,\nu) \leq C(d,R)\,W_2^{1/4}(\mu,\nu)\,,
\end{align}
where $C(d,R)$ is an explicit constant involving the dimension and the radius.\footnote{For simplicity, we stated the result as holding for measures over a ball of radius $R$, though an analogous result holds when the target measures are supported on a convex, compact subset of $\R^d$.} {We note that the Hölder exponent in the first inequality for the $W_1$ distance is also sharp due to \cite{letrouit_unstable_2026}.}

\paragraph{Stability of entropic transport maps.}
Entropic transport maps are a computationally friendly proxy for estimating optimal transport maps \citep{pooladian2021entropic,pooladian2023minimax,rigollet2025sample}. Entropic transport maps are governed by a regularization parameter $\eps > 0$, written
\begin{align*}
    \mathsf  T_\eps^{\rho\to\mu}(x) \defeq \E_{\pi_\eps^{\rho\mu}}[Y\mid X=x]\,,
\end{align*}
where $\pi_\eps^{\rho\mu}$ is the entropic optimal coupling between $\rho$ and $\mu$ \citep{peyre2019computational, pooladian2021entropic}. Recently, a tight stability bound was established by \cite{divol2025tight}: for any $\rho,\mu,\nu \in \cP(B(0,R))$, 
\begin{align}\label{eq:eot_stab}
    \|\mathsf  T_\eps^{\rho\to\mu} - \mathsf  T_\eps^{\rho\to\nu}\|_{L^2(\rho)}\leq \frac{2R^2}{\eps}\, W_2(\mu,\nu)\,.
\end{align}
This result turns out to be an exponential improvement on prior work by \cite{carlier2024displacement}, who proved this result where the constant scales as $\exp(\eps^{-1})$. Equation~\eqref{eq:eot_stab} implies that the embedding $\mu \mapsto \mathsf  T_\eps^{\rho\to\mu}$ is Lipschitz in $W_2$, which improves upon the $\tfrac14$-H\"older continuity of optimal transport maps (recall \eqref{eq:one_over_four}). Finally, it is worth mentioning that under suitable conditions~\citep[see, e.g.,][]{chewi2023entropic}, the results of \cite{divol2025tight} can be used to recover \eqref{eq:ot_lip_stab} (up to a universal constant) by taking the $\eps \searrow 0$ limit.

\section{Main results}

\subsection{ODE representation and starting point}\label{sec:ode_rep}
We follow the coupling construction of~\citet{chewi2025stability}, but with a notable refinement. Namely,~\citet{chewi2025stability} considered the ODE~\eqref{eq:rev_ou_ode} over a finite time horizon $[0,T]$ and then let $T \to \infty$. However, in Lemma~\ref{lem:km_regular_case}, we show that for $\mu$, $\nu$ satisfying~\eqref{theta}, the $T\to\infty$ limit directly admits an ODE interpretation. Namely, we have the representation
\begin{align}\label{eq:ode_representation}
    \KM^\mu_{\delta\gets} = {\id} + \int_\delta^\infty \nabla \log \frac{q_t^\mu}{\gamma} \circ \KM^\mu_{t\gets}\dd t
\end{align}
for each $\delta > 0$.
Then, we can take $\delta \searrow 0$ using Lemma~\ref{lem:existence_km}.
This leads to the following analogue of~\eqref{eq:stab_fisher_old} in~\cite{chewi2025stability}.

\begin{proposition}
    Suppose that $\mu,\nu \in \cP_2(\R^d)$ satisfy~\eqref{theta}, and let $(\theta_t)_{t\ge 0}$ denote the corresponding constants for $\mu$.
    Then,
\begin{align}\label{eq:stab_start}
    \frac{\dd}{\dd t} \, \|\KM^\mu_{t\gets} - \KM^\nu_{t\gets}\|_{L^2(\gamma)}\geq -\theta_t\,\|\KM^\mu_{t\gets} - \KM^\nu_{t\gets}\|_{L^2(\gamma)} - \sqrt{\FI(q^\nu_t\mmid q^\mu_t)}\,.
\end{align}
\end{proposition}
\begin{proof}
    Indeed,
    \begin{align*}
        \frac{\dd}{\dd t} \,\|\KM^\mu_{t\gets} - \KM^\nu_{t\gets}\|_{L^2(\gamma)}^2
        &= -2\int_{\RR^d}\Bigl\langle \KM^\mu_{t\gets} - \KM^\nu_{t\gets}\,,\, \nabla \log \frac{q_t^\mu}{\gamma} \circ \KM^\mu_{t\gets} - \nabla \log \frac{q_t^\nu}{\gamma} \circ \KM^\nu_{t\gets}\Bigr\rangle\dd \gamma \\[0.25em]
        &= -2\int_{\RR^d}\Bigl\langle \KM^\mu_{t\gets} - \KM^\nu_{t\gets}\,,\, \nabla \log \frac{q_t^\mu}{\gamma} \circ \KM^\mu_{t\gets} - \nabla \log \frac{q_t^\mu}{\gamma} \circ \KM^\nu_{t\gets}\Bigr\rangle\dd \gamma \\[0.25em]
        &\qquad{} -2\int_{\RR^d}\Bigl\langle \KM^\mu_{t\gets} - \KM^\nu_{t\gets}\,,\, \bigl(\nabla \log \frac{q_t^\mu}{\gamma} - \nabla \log \frac{q_t^\nu}{\gamma}\bigr) \circ \KM^\nu_{t\gets}\Bigr\rangle\dd \gamma \\
        &\ge -2\theta_t\,\|\KM^\mu_{t\gets} - \KM^\nu_{t\gets}\|_{L^2(\gamma)}^2 - 2\,\|\KM^\mu_{t\gets} - \KM^\nu_{t\gets}\|_{L^2(\gamma)} \sqrt{\FI(q_t^\nu \mmid q_t^\mu)}\,.
    \end{align*}
    This implies the result.
\end{proof}
Starting from this proposition, we use suitable bounds on the Fisher information between the time marginals of the OU flow in order to establish stability.

A well-known result is de Bruijn's identity along OU channels, which we will use repeatedly. We state it here for reference.
\begin{lemma}
For any $t > 0$, it holds that
\begin{align*}
    \frac{\dd}{\dd t}\,{\mathsf H}(q_t^\nu\mmid q_t^\mu) = -\FI(q_t^\nu\mmid q_t^\mu)\,.
\end{align*}
\end{lemma}

\subsection{Warm-up: Stability in relative entropy}
Leveraging the log-Sobolev inequality \eqref{eq:lsi}, we obtain almost immediately a stability bound in relative entropy.

\begin{theorem}\label{thm:stab_ent}
For $\mu \in \cP_2(\mathbb{R}^d)$ which satisfies \eqref{theta}, and any $\nu \in \cP_2(\mathbb{R}^d)$, it holds that
\begin{align}\label{eq:ent_stab_bd}
\|\mathsf T^{\mu} - \mathsf T^{\nu}\|_{L^2(\gamma)}^2 \leq 2L_0^2(\mu) \, {\mathsf H}(\nu\mmid\mu)\,,
\end{align}
where $$L_0(\mu) \deq \exp\Bigl(\int_0^\infty \theta_t \dd t\Bigr)\, .$$
\end{theorem}

Notice that this immediately provides a strengthening of Talagrand's $T_2$-inequality by introducing a tighter lower bound on the relative entropy, as
\begin{align*}
    \tfrac{1}{2}W_2^2(\mu,\nu) \leq \tfrac12 \|\mathsf T^{\mu} - \mathsf T^{\nu}\|_{L^2(\gamma)}^2 \leq L_{0}^2(\mu) \, {\mathsf H}(\nu\mmid\mu)\,,
\end{align*}
where the leftmost inequality follows from a trivial coupling argument. Also note that Theorem~\ref{thm:stab_ent} implies (and sharpens) the original stability bound by \cite{chewi2025stability} by applying the log-Sobolev inequality, which is satisfied for $\mu$ thanks to \eqref{theta}. 

\begin{proof}[Proof of Theorem~\ref{thm:stab_ent}]
First, consider $\nu$ which also satisfies~\eqref{theta}. Below, $(\theta_t)_{t\ge 0}$ denotes the corresponding constants for $\mu$ only.

Starting from \eqref{eq:stab_start}, we directly apply Gr\"onwall's inequality  on the interval $[\delta,T]$ to obtain
\begin{align*}
    \|\mathsf T_{T\gets}^\mu - \mathsf T_{T\gets}^\nu\|_{L^2(\gamma)} &\geq \exp\Bigl(-\int_\delta^T \theta_s \dd s\Bigr)\,\|\mathsf T_{\delta\gets}^\mu - \mathsf T_{\delta\gets}^\nu\|_{L^2(\gamma)} 
    \\[0.25em]
    &\qquad{}- \int_\delta^T \exp\Bigl(-\int_s^T \theta_{r}\dd r\Bigr)\FI(q^\nu_{s}\mmid q^\mu_{s})^{1/2} \dd s\,.
\end{align*}
Letting $T \rightarrow \infty$ then gives
\begin{align}\label{eq:stab_FIs}
     \|\mathsf T_{\delta\gets}^{\mu} - \mathsf T_{\delta\gets}^{\nu}\|_{L^2(\gamma)} \leq L_\delta(\mu) \int_\delta^\infty \exp\Bigl(-\int_s^\infty \theta_{r}\dd r\Bigr) \FI(q^\nu_{s}\mmid q^\mu_{s})^{1/2}\dd s \,,
\end{align}
where $L_\delta(\mu)$ has been defined in \eqref{eq:def-L_t}.
We now claim that
\begin{align*}
    \FI(q^\nu_{t}\mmid q^\mu_{t})^{1/2} \leq -\sqrt{2}L_t(\mu)\, \frac{\dd}{\dd t}\, {\mathsf H}(q^\nu_{t}\mmid q^\mu_{t})^{1/2}\,.
\end{align*}
To prove it, observe that since $q_{t}^\mu = (\mathsf T^{\mu}_{t\gets})_{\sharp}\gamma$ with $\mathsf T^{\mu}_{t\gets}$ being Lipschitz continuous with constant $L_{t}(\mu)$,
$q_{t}^\mu$ satisfies the following log-Sobolev inequality (recall \eqref{eq:lsi}):
\begin{align*}
{\mathsf H}(q_{t}^\nu\mmid q_{t}^\mu) \leq \frac{L^2_{t}(\mu)}{2}\FI(q_{t}^\nu\mmid q_{t}^\mu)\,.
\end{align*}
Using de Bruijn's identity and the above log-Sobolev inequality, it holds that
\begin{align*}
    -\sqrt{2}L_t\, \frac{\dd}{\dd t}\, {\mathsf H}(q^\nu_{t}\mmid q^\mu_{t})^{1/2} 
    = L_t\, \frac{\FI(q_{t}^\nu\mmid q_{t}^\mu)}{\sqrt{2}\,{\mathsf H}(q^\nu_{t}\mmid q^\mu_{t})^{1/2} }
    \geq \FI(q^\nu_{t}\mmid q^\mu_{t})^{1/2}\,,
\end{align*}
which is the claim.

Substituting the claim into \eqref{eq:stab_FIs} then yields
\begin{align*}
     \|\mathsf T_{\delta\gets}^{\mu} - \mathsf T_{\delta\gets}^{\nu}\|_{L^2(\gamma)} &\leq -\sqrt{2}L_\delta(\mu) \int_\delta^\infty \exp\Bigl(-\int_t^\infty \theta_s \dd s\Bigr) \exp\Bigl(\int_t^\infty \theta_s \dd s\Bigr)\,\frac{\dd}{\dd t}\,{\mathsf H}(q_t^\nu\mmid q_t^\mu)^{1/2} \dd t \\
    &= -\sqrt{2}L_\delta(\mu)\int_\delta^\infty \frac{\dd }{\dd t}\,\mathsf H(q_t^\nu\mmid q_t^\mu)^{1/2}\dd t\\
    &= \sqrt{2} L_\delta(\mu)\,\bigl( {\mathsf H}(q_\delta^\nu\mmid q_\delta^\mu)^{1/2} - {\mathsf H}(\gamma\mmid \gamma)^{1/2}\bigr)\,.
\end{align*}
Now, for $\nu \in \cP_2(\R^d)$, let $\{\nu_n\}_{n\in\N}$ be a regular approximating sequence for $\nu$ as in Definition~\ref{def:reg_approx}.
Applying the above bound to $\nu_n$,
\begin{align*}
     \|\mathsf T_{\delta\gets}^{\mu} - \mathsf T_{\delta\gets}^{\nu_n}\|_{L^2(\gamma)}
    \le \sqrt{2} L_\delta(\mu)\,{\mathsf H}(q_\delta^{\nu_n}\mmid q_\delta^\mu)^{1/2}\,.
\end{align*}
As $n\to\infty$, the LHS converges to $\|\KM^\mu_{\delta\gets} - \KM^\nu_{\delta\gets}\|_{L^2(\gamma)}$, by Lemma~\ref{lem:existence_km}.
As for the RHS, as shown in~\citet[Corollary 4]{PolWu16WassCont}, the heat flow (or equivalently, the OU process) induces $W_2$ continuity of the relative entropy, and thus $\KL(q_\delta^{\nu_n} \mmid q_\delta^\mu) \to \KL(q_\delta^\nu \mmid q_\delta^\mu)$.
This yields
\begin{align*}
     \|\mathsf T_{\delta\gets}^{\mu} - \mathsf T_{\delta\gets}^{\nu}\|_{L^2(\gamma)}
    \le \sqrt{2} L_\delta(\mu)\,{\mathsf H}(q_\delta^{\nu}\mmid q_\delta^\mu)^{1/2}
    \le \sqrt{2} L_\delta(\mu)\,{\mathsf H}(\nu\mmid \mu)^{1/2}\,,
\end{align*}
by the data-processing inequality.
Finally, let $\delta\searrow 0$ using Lemma~\ref{lem:existence_km} again.
\end{proof}

\begin{remark}
Our initial proof of Theorem \ref{thm:stab_ent} is inspired by the proof of Talagrand's inequality in \cite{otto2000generalization} using the OU flow. In fact, carefully tracing their proof, one can recover the result of Theorem \ref{thm:stab_ent} for $\mu = \gamma$ (and hence $\theta_s =0$). This proof (included in Appendix~\ref{sec:proof_old} for completeness) generalizes their Lyapunov function to accommodate for $\mu$ satisfying \eqref{theta}, which might be of independent interest. Using the log-Sobolev inequality at the right moment yields the following differential inequality
\begin{align*}
&\frac{\dd}{\dd t} \bigl[ \|\KM^\mu_{t\gets} - \KM^\nu_{t\gets}\|_{L^2(\gamma)} - (2L^2_{t}(\mu)\,{\mathsf H}(q^\nu_{t}\mmid q^\mu_{t}))^{1/2} \bigr] \\
&\qquad \qquad \geq -\theta_{t}\,\bigl[\|\KM^\mu_{t\gets} - \KM^\nu_{t\gets}\|_{L^2(\gamma)} - (2L^2_{t}(\mu)\,{\mathsf H}(q^\nu_{t}\mmid q^\mu_{t}))^{1/2}\bigr]\,,
\end{align*}

and one concludes with Grönwall's inequality and $T \rightarrow \infty$.
\end{remark}

\subsection{Near-Lipschitz stability in \texorpdfstring{$2$}{2}-Wasserstein}

Now we leverage the regularizing effect of the heat flow to deduce stability of the Kim--Milman flow map with respect to the $2$-Wasserstein distance. Our goal is to prove the following result.

\begin{theorem}[Main theorem]\label{thm:wass_stab}
    Suppose $\mu$ satisfies \eqref{theta} and $\nu \in \cP_2(\R^d)$ with $\mathsf m_2(\mu),\mathsf m_2(\nu) \leq M$. In addition, let $\mathsf C_\mu \geq 0$ be such that for all $0\leq s\leq t$, it holds that
    \begin{align}\label{eq:int_cond}
        \int_s^t \theta_u \dd u \leq \mathsf C_\mu\,.
    \end{align}
    Then,
    \begin{align*}
        \|\mathsf T^{\mu} - \mathsf T^{\nu}\|_{L^2(\gamma)} \leq
        W_2(\mu,\nu)\exp(3\mathsf C_\mu) \log\Bigl[e^2\max\Bigl\{1,\, \frac{4M^2 + 8d}{W_2^2(\mu,\nu)\exp(6\mathsf C_\mu)}\Bigr\}\Bigr]\,.
    \end{align*}
\end{theorem} 

We give some remarks before continuing to the proof.

In known cases where $\mu$ satisfies \eqref{theta}, condition \eqref{eq:int_cond} can be directly inferred; see Section~\ref{sec:ex} below for examples. 

We emphasize that \eqref{theta} does not imply bi-Lipschitzness of $\mathsf T^{\mu}$ but merely Lipschitzness. It is perhaps surprising that nearly Lipschitz stability in $2$-Wasserstein is possible under this assumption, given that the closest result for optimal transport maps requires bi-Lipschitzness.

As for the log-factor, we believe this is an artifact of our proof technique and do not believe it is necessary. See Section~\ref{sec:lower_bound} for more discussions on this point.

\subsubsection{Proof via early stopping}

To prove stability in relative entropy, we leveraged the log-Sobolev inequality in order to suitably bound the time-integrated (square-root of the) Fisher information by the relative entropy. In order to obtain stability bounds in Wasserstein distance, it is hence natural to use the regularizing effect along the OU semigroup comparing Fisher information and Wasserstein distance. This is captured in the following lemma.
\begin{lemma}\label{lem:W2-reg-OU}
Let $\mu, \nu \in \cP_2(\RR^d)$. Then
\begin{itemize}
\item For any $t > 0$,
\begin{equation}\label{eq:H<W2}
    {\mathsf H}(q_{t}^\nu \mmid q_{t}^\mu) \leq \frac{1}{2(e^{2t}-1)}\, W_2^2(\nu,\mu)\,.
\end{equation}
\item If in addition $\mu$ satisfies \eqref{theta}, then for any $t > t_0 > 0$,
\begin{equation}\label{eq:FI<W2}
    {\FI}(q_{t}^\nu\mmid q_{t}^\mu) \le \frac{1}{c(t)}\, W_2^2(\nu,\mu)\,,
\end{equation}
where we define
\begin{align}\label{eq:ct}
    c(t) = 2(e^{2t_0}-1) \int_{t_0}^t \exp\Bigl(\int_s^t 2(1-2\theta_r)\dd r\Bigr)\dd s\,.
\end{align}
\end{itemize}
\end{lemma}

\begin{proof}
The first bound corresponds to the Harnack inequality introduced by \cite{wang1997logarithmic}; more precisely, it is the log-Harnack inequality due to~\citet{BobGenLed01Hyper}. See also~\citet[Corollary 3.11]{altschuler_shifted_2025} for the version we use.

For the second bound, we use the following result from {\citet[Lemma 2]{wibisono2025mixing}} 
to give an upper bound on the change in relative Fisher information along OU channels: For any $t > 0$, it holds that
    \begin{align*}
        \frac{\dd}{\dd t}\FI(q^\nu_t\mmid q^\mu_t) \leq - 2\,\E_{q_t^\nu}\Bigl\langle \nabla \log \frac{q_t^\nu}{q_t^\mu}\,,\,(-2\nabla^2 \log q_t^\mu - I)\,\nabla \log \frac{q_t^\nu}{q_t^\mu} \Bigr\rangle\,.
    \end{align*}
In particular, if $\mu$ is such that \eqref{theta} holds, then an immediate corollary is the bound
\begin{align}\label{eq:andre_corollary}
\frac{\dd}{\dd t} \FI(q_t^\nu\mmid q_t^\mu) \leq -2(1-2\theta_t)\FI(q_t^\nu\mmid q_t^\mu)\,.
\end{align}
Gr\"onwall's inequality on $[s,t]$, $0 < s \leq t$, then gives
\begin{align*}
    {\FI}(q_{s}^\nu\mmid q_{s}^\mu)
    \ge {\FI}(q_t^\nu\mmid q_{t}^\mu) \exp\Bigl(\int_s^t 2(1-2\theta_r)\dd r\Bigr)\,.
\end{align*}
Then, for $0 < t_0 < t$, we have by de Bruijn's identity
\begin{align*}
    {\mathsf H}(q_{t_0}^\nu\mmid q_{t_0}^\mu) &\geq {\mathsf H}(q_{t_0}^\nu\mmid q_{t_0}^\mu) - {\mathsf H}(q_{t}^\nu\mmid q_{t}^\mu) \\
    &=  \int_{t_0}^t {\FI}(q_{s}^\nu\mmid q_{s}^\mu)\dd s \\
    &\ge {\FI}(q_{t}^\nu\mmid q_{t}^\mu) \int_{t_0}^t \exp\Bigl(\int_s^t 2(1-2\theta_r)\dd r\Bigr)\dd s\,.
\end{align*}
We conclude by rearranging and using \eqref{eq:H<W2}.
\end{proof}

Ideally, we start from \eqref{eq:stab_start} and use \eqref{eq:FI<W2}. Unfortunately, the right-hand side of \eqref{eq:FI<W2} is not integrable at $0$, albeit barely. Thus, we require  a different approach, based on ``early stopping''.
For $0 < \varepsilon \le \delta$, we consider
\begin{align*}
    \|\KM^\mu_{\eps\gets} -\KM^\nu_{\eps\gets}\|_{L^2(\gamma)} = \|\KM^\mu_{\delta\gets} - \KM^\nu_{\delta\gets}\|_{L^2(\gamma)} - \int_\eps^\delta \frac{\dd}{\dd t}\, \|\KM^\mu_{t\gets} - \KM^\nu_{t\gets}\|_{L^2(\gamma)} \dd t\,,
\end{align*}
and later perform a trade-off to determine the choice of $\delta > 0$. 
Then, we send $\varepsilon\searrow 0$.

We bound the first term in a straightforward manner by employing the regularization of the Fisher information by the heat flow.
In what follows, $(\theta_t)_{t\ge 0}$ always refers to the regularity constants for $\mu$ (not for $\nu$).

\begin{proposition}\label{prop:ineq_1}
Suppose that $\mu$, $\nu$ satisfy \eqref{theta}.
Then,
\begin{align*}
    \|\KM^\mu_{\delta\gets} - \KM^\nu_{\delta\gets}\|_{L^2(\gamma)}
    &\leq W_2(\mu,\nu)\int_\delta^\infty \frac{\exp\bigl(\textstyle \int_\delta^t \theta_\omega \dd \omega\bigr)}{\sqrt{c(t)}}\dd t\,,
\end{align*}
where $c(t)$ is defined in \eqref{eq:ct}.
\end{proposition}
\begin{proof}
By~\eqref{eq:stab_FIs} in the proof of Theorem~\ref{thm:stab_ent},
\begin{align*}
     \|\mathsf T_{\delta\gets}^{\mu} - \mathsf T_{\delta\gets}^{\nu}\|_{L^2(\gamma)} \leq L_\delta(\mu) \int_\delta^\infty \exp\Bigl(-\int_t^\infty \theta_{r}\dd r\Bigr) \FI(q^\nu_{t}\mmid q^\mu_{t})^{1/2}\dd t \,.
\end{align*}
We conclude by employing the second point from Lemma \ref{lem:W2-reg-OU}.
\end{proof}

For the remaining interval, we obtain the following explicit bound.

\begin{proposition}\label{prop:ineq_2}
For any $\mu,\nu$ satisfying~\eqref{theta}, assuming that $\mom(\mu),\mom(\nu)\leq M$,
\begin{align*}
    - \int_\eps^\delta \frac{\dd}{\dd t}\, \|\KM^\mu_{t\gets} - \KM^\nu_{t\gets}\|_{L^2(\gamma)} \dd t
    &\le 2M\,(1-e^{-\delta}) + 2\sqrt{d\,(1-e^{-2\delta})}\,.
\end{align*}
\end{proposition}
\begin{proof}
From \eqref{eq:ode_representation}, we have by Cauchy--Schwarz
\begin{align*}
    -\frac{\dd}{\dd t}\, \|\KM^\mu_{t\gets} - \KM^\nu_{t\gets}\|_{L^2(\gamma)}^2
    &=2\int\Bigl\langle \KM^\mu_{t\gets} - \KM^\nu_{t\gets}\,,\, \nabla \log \frac{q_t^\mu}{\gamma}\circ \KM^\mu_{t\gets} - \nabla \log \frac{q_t^\nu}{\gamma}\circ \KM^\nu_{t\gets} \Bigr\rangle\dd \gamma \\
    &\leq 2\, \|\KM^\mu_{t\gets} - \KM^\nu_{t\gets}\|_{L^2(\gamma)}\,\bigl(\sqrt{\FI(q_t^\mu\mmid \gamma)} + \sqrt{\FI(q_t^\nu\mmid \gamma)}\bigr)
\end{align*}
and thus
\begin{align*}
    -\frac{\dd}{\dd t}\, \|\KM^\mu_{t\gets} - \KM^\nu_{t\gets}\|_{L^2(\gamma)}
    &\le \sqrt{\FI(q_t^\mu\mmid \gamma)} + \sqrt{\FI(q_t^\nu\mmid \gamma)}\,.
\end{align*}
By Lemma \ref{lem:FI_abs_bound}, 
\begin{align*}
    \sqrt{\FI(q_{t}^\mu\mmid\gamma)}\vee \sqrt{\FI(q_{t}^\nu\mmid\gamma)} \leq Me^{-t} + \frac{\sqrt d e^{-2t}}{\sqrt{1-e^{-2t}}}\,.
\end{align*}
We conclude by integration, as
\begin{align*}
\int_{0}^\delta  e^{-2t}\, (1-e^{-2t})^{-1/2}\dd t
&=  \sqrt{1-e^{-2\delta}}\,. \qedhere
\end{align*}
\qedhere
\end{proof}

With these two inequalities, we can prove our main result.

\begin{proof}[Proof of Theorem~\ref{thm:wass_stab}]
Combining both Proposition~\ref{prop:ineq_1} and Proposition~\ref{prop:ineq_2}, we have that along a regular approximating sequence $\{\nu_n\}_{n\in\N}$, for any $0 < \eps \le \delta$,
\begin{align*}
    \|\KM^\mu_{\eps\gets} - \KM^{\nu_n}_{\eps\gets}\|_{L^2(\gamma)}
    &\le W_2(\mu,\nu_n)\int_\delta^\infty \frac{\exp\bigl(\textstyle \int_\delta^t \theta_\omega \dd \omega\bigr)}{\sqrt{c(t)}}\dd t + 2M_n\,(1-e^{-\delta}) + 2\sqrt{d\,(1-e^{-2\delta})}\,,
\end{align*}
where $M_n \deq \max\{\mom(\mu), \mom(\nu_n)\}$.
Passing $n\to\infty$ and then $\eps\searrow 0$ via Lemma~\ref{lem:existence_km},
\begin{align*}
    \|\KM^\mu - \KM^{\nu}\|_{L^2(\gamma)}
    &\le W_2(\mu,\nu)\int_\delta^\infty \frac{\exp\bigl(\textstyle \int_\delta^t \theta_\omega \dd \omega\bigr)}{\sqrt{c(t)}}\dd t + 2M\,(1-e^{-\delta}) + 2\sqrt{d\,(1-e^{-2\delta})}\,,
\end{align*}
with $M \deq \max\{\mom(\mu), \mom(\nu)\}$.
We now turn to bound $c(t)$. By \eqref{eq:int_cond} and the choice $t_0 = t/2$, we see that
\begin{align*}
    \int_{t/2}^t \exp\Bigl(\int_s^t 2(1-2\theta_r) \dd r \Bigr)\dd s &\geq \exp(-4\mathsf C_\mu) \int_{t/2}^t e^{2(t-s)}\dd s \\
    &= \exp(-4\mathsf C_\mu)\,  \frac{e^{2t}}{2}\,(e^{-t} - e^{-2t}) \\
    &= \frac{1}{2\exp(4\mathsf C_\mu)}\, (e^t - 1)\,,
\end{align*}
and so we have the bound $c(t) \geq \exp(-4\mathsf C_\mu)\,(e^{t}-1)^{2}$. 
Rearranging and invoking \eqref{eq:int_cond} again, our bound now reads
\begin{align*}
    \|\mathsf T^{\mu} - \mathsf T^{\nu}\|_{L^2(\gamma)} \leq  W_2(\mu,\nu) \exp(3\mathsf C_\mu) \int_\delta^\infty (e^t-1)^{-1} \dd t + 2M\,(1-e^{-\delta}) + 2\sqrt{d\,(1-e^{-2\delta})}\,.
\end{align*}
We also readily compute
\begin{align*}
    \int_\delta^\infty (e^t - 1)^{-1} \dd t &= - \log(1-e^{-\delta})\,.
\end{align*}
Writing $\zeta \deq 1-e^{-\delta}$, our bound reads
\begin{align*}
    \|\mathsf T^{\mu} - \mathsf T^{\nu}\|_{L^2(\gamma)} \leq  W_2(\mu,\nu) \exp(3\mathsf C_\mu) \log(1/\zeta) + 2M\zeta + \sqrt{8d\zeta}\,.
\end{align*}
A simple choice is to make the last two terms at most $W_2(\mu,\nu)\exp(3\mathsf C_\mu)$, by taking the choice $\zeta \deq 1 \wedge \frac{W_2(\mu,\nu) \exp(3\mathsf C_\mu)}{2M} \wedge \frac{W_2^2(\mu,\nu) \exp(6\mathsf C_\mu)}{8d}$.
This yields
\begin{align*}
    &W_2(\mu,\nu) \exp(3\mathsf C_\mu)\,(2 + \log(1/\zeta)) \\
    &\qquad = W_2(\mu,\nu) \exp(3\mathsf C_\mu) \log\Bigl[e^2\max\Bigl\{1,\, \frac{2M}{W_2(\mu,\nu)\exp(3\mathsf C_\mu)},\, \frac{8d}{W_2^2(\mu,\nu)\exp(6\mathsf C_\mu)}\Bigr\}\Bigr] \\[0.25em]
    &\qquad \le W_2(\mu,\nu) \exp(3\mathsf C_\mu) \log\Bigl[e^2\max\Bigl\{1,\, \frac{4M^2 + 8d}{W_2^2(\mu,\nu)\exp(6\mathsf C_\mu)}\Bigr\}\Bigr]\,,
\end{align*}
where we used $A \le \frac{1}{2}\,(1+A^2) \le \max\{1,A^2\}$ to simplify the bound.
\end{proof}

\begin{remark}
    In the statement of the informal theorem, we can use $M \le \mom(\mu) + W_2(\mu,\nu)$ to show that the term in the logarithm can indeed be written as $1+c'(\mu)/W_2(\mu,\nu)$ for some $c'(\mu)$ that only depends on $\mu$.
\end{remark}
\subsection{Some examples}\label{sec:ex}

We now go over some cases where \eqref{theta} has been previously established. In all instances, we will write $\mu \propto \exp(-V)$ for some function $V :\R^d \to \R$, and state the assumption on $\mu$ in terms of properties of $V$.

The most general hypothesis for validity of \eqref{theta} has been established in \cite{chaintron_propagation_2025}, more precisely in the proof of Theorem 1.4 therein. It is commonly referred to as asymptotic convexity, a hypothesis well-known in the coupling by reflection literature, ensuring among other things exponential contraction in the $1$-Wasserstein distance of the associated overdamped Langevin dynamics, see, e.g., \cite{eberle2016reflection}. In order to state rigorously what we mean by asymptotic convexity, we define the convexity profile of $V$ as follows
\begin{align*}
    \kappa_V(r) \defeq \inf\Bigl\{ \frac{\langle \nabla {V}(x) - \nabla {V}(y), x-y \rangle}{\|x-y\|^2} \, : \, \|x-y\| = r\Bigr\} \quad r>0\, .
\end{align*}
The potential $V$ is then said to be $\alpha$-asymptotically convex, $\alpha>0$, if
\begin{align*}
    \liminf_{r \rightarrow + \infty} \kappa_V(r) \geq \alpha \quad \text{and} \quad \int_0^1 \kappa_V^-(r)\, r \dd r < + \infty \,,
\end{align*}
where $\kappa_V^-(r) := \max\{0,-\kappa_V(r)\}$ denotes the negative part of $\kappa_V(r)$.

Obviously, this class encompasses $\alpha$-strongly convex potentials $V$, as in this case $\kappa_V(r) \geq \alpha$ for all $r > 0$.

Under the assumptions of asymptotic convexity, by properly rescaling \cite[Proposition 2.9]{chaintron_propagation_2025} one has the following expression
\[
\theta_t = \frac{1-\alpha}{\alpha(e^{2t}-1)+1} -\frac{e^{2t}}{(e^{2t}-1)(\alpha(e^{2t}-1)+1)} +\frac{e^{2t}}{8 (e^{2t}-1)^2}\, \mathbb{E}[X_t^2]\,,
\]
where $X_t$ is a random variable with density
$$\xi_t ({\rm d}x) \propto \mathds{1}_{x \geq 0} \exp \Bigl( -\frac{1}{8} \int_0^{x} u\kappa_{V}^- (u) \dd u \Bigr) \exp \Bigl( -\frac{x^2}{16 (1-e^{-2t})} \Bigr)\,. $$

Let us give two more explicit examples of asymptotically convex potentials, where the constants are tractable.

\paragraph{Strongly convex with Lipschitz perturbation.}
Suppose $V = W + H$ where $W$ is $\alpha$-strongly convex and $H$ is a smooth $L$-Lipschitz perturbation. In this case
$
\kappa_V(r) \geq \alpha - 2L/r
$
for $r\geq 0$, so $V$ is asymptotically convex.
\cite{BriPed25HeatFlow}, see also \cite{chaintron_propagation_2025} for a slight improvement of some absolute constants, showed that \eqref{theta} holds with 
\begin{align}\label{eq:slc_pert_theta}
    \theta_u = \frac{1-\alpha}{\alpha(e^{2u}-1)+1} + \frac{e^{2u}L^2}{(\alpha(e^{2u}-1)+1)^2}+ \frac{2Le^{2u}}{(\alpha(e^{2u}-1)+1)^{3/2}\sqrt{e^{2u}-1}}\,.
\end{align}
In the case where $\alpha \in (0,1]$, all these terms are positive and so we have the na\"{\i}ve bound
\begin{align*}
    \int_s^t \theta_u \dd u \leq \int_0^\infty \theta_u \dd u = \frac{1}{2}\log(\alpha^{-1}) + \frac{L^2}{2\alpha} + \frac{2L}{\sqrt{\alpha}}\,.
\end{align*}
If $\alpha > 1$, we can always drop the first term in \eqref{eq:slc_pert_theta}, and obtain
\begin{align*}
    \int_s^t \theta_u \dd u \leq \frac{L^2}{2\alpha} + \frac{2L}{\sqrt{\alpha}}\,,
\end{align*}
which is the required bound for \eqref{eq:int_cond}. 
\paragraph{Strongly convex with Lipschitz perturbation with bounded Hessian.}
If now $V = W + H$ where $W$ is strongly convex and $H$ is a smooth Lipschitz function with bounded Hessian, then one can find suitable $\alpha, M>0$ such that
$$\kappa_V(r) \geq \alpha - \frac{1}{r}\,f_M(r)\, ,$$
where $f_M(r)  \deq  2\sqrt{M} \tanh(r\sqrt{M}/2 )$.
Studying propagation of convexity profiles of this type has been pioneered by \cite{conforti2024weak}, and further used in several works such as \cite{conforti2025projected}. We cite the formula for $\theta_t$ from \citet[Lemma B.3]{silveri2025beyond}
\begin{align*}
    \theta_t &= 1 - \frac{\alpha}{\alpha+(1-\alpha)e^{-2t}} + \frac{e^{-2t}}{(\alpha+(1-\alpha)e^{-2t})^2}\,M \\
    &= \frac{1-\alpha}{\alpha(e^{2t}-1)+1} + \frac{e^{2t}M}{(\alpha(e^{2t}-1)+1)^2} \,.
\end{align*}
As before, it is now immediate that \eqref{theta} and \eqref{eq:int_cond} are satisfied.

\subsection{Fisher information lower bound}\label{sec:lower_bound}
The core of our proof hinges on \eqref{eq:FI<W2}, which upper bounds the relative Fisher information along OU channels by the $2$-Wasserstein distance. In this section, we prove that this Fisher information bound cannot be improved by more than a constant factor in general. By rescaling, it is sufficient to consider the example for the heat semigroup, see also the remark below.

Henceforth, for a probability measure $P \in \cP(\R)$, we write $P_t \defeq P * \cN(0,t)$. We will show the following.
\begin{proposition}\label{prop:fi_lb}
Fix $t > 0$ and let $\mu = \cN(0,1)$; it holds that
\begin{align*}
    \sup_{\nu \neq \mu} \frac{\FI(\nu_t\mmid \mu_t)}{W_2^2(\mu,\nu)} \geq \frac{1}{t^2}\,.
\end{align*}
\end{proposition}
To prove this result, we will construct for each $t > 0$, a $\nu \in \cP(\R)$ that verifies the lower bound. To this end, let $\mu = \cN(0,1)$ and fix $t > 0$.
Let $\mu_t = \mu * \cN(0,t) = \cN(0,1+t)$, $\nu_t = \nu * \cN(0,t)$. Letting $R > 0$ (which will ultimately be taken to $+\infty$), define the intervals
\begin{align*}
    A_R \defeq [R/(1+t), R/(1+t) + R^{-1}]\,, \quad B_R \defeq [R - R^{1/3}, R + R^{1/3}]\,.
\end{align*}
Our adversarial choice of $\nu$ is given by
\begin{align*}
    \nu \defeq \mu - \mu\big|_{A_R} + p_R \delta_R\,,
\end{align*}
where $p_R \defeq \mu(A_R)$. Note that $\nu \geq p_R \delta_R$ always holds.

First, we bound the $2$-Wasserstein distance from above. Choosing the coupling to only move mass from inside the interval $A_R$ to the point-mass $\delta_R$, we obtain
\begin{align}\label{eq:counterex_0}
    W_2^2(\mu,\nu) \leq \int_{A_R}|x-R|^2 \,\mu(\!\dd x) = \frac{p_Rt^2 R^2}{(1+t)^2}\,(1+o(1))\,,
\end{align}
where the asymptotics are for $R\to\infty$, with $t$ fixed. If we then manage to show that
\begin{align}\label{eq:fi_w2_proofstep}
    \FI(\nu_t\mmid \mu_t) \geq \frac{p_RR^2}{(1+t)^2}\,(1-o(1))\,,
\end{align}
the proof concludes, as
\begin{align*}
    \frac{\FI(\nu_t\mmid \mu_t)}{W_2^2(\mu,\nu)} \geq \frac{(p_R R^2\,(1+t)^{-2})\,(1-o(1))}{(p_R R^2\,(1+t)^{-2}\,t^2)\,(1+o(1))} =\frac{1-o(1)}{t^2\,(1+o(1))}\,, 
\end{align*}
and hence
\begin{align*}
    \limsup_{R\to\infty} \frac{\FI(\nu_t\mmid\mu_t)}{W_2^2(\mu,\nu)} \ge \frac{1}{t^2}\,.
\end{align*}
To obtain \eqref{eq:fi_w2_proofstep}, it is sufficient to study the integral over $B_R$
\begin{align*}
    \FI(\nu_t\mmid \mu_t)
    &= \int \bigl\lvert \bigl(\log \frac{\nu_t}{\mu_t}\bigr)'\bigr\rvert^2\dd \nu_t
    \ge \int_{B_R} \bigl\lvert \bigl(\log \frac{\nu_t}{\mu_t}\bigr)'\bigr\rvert^2\dd \nu_t
\end{align*}
where we conclude thanks to the following two bounds
\begin{align}\label{eq:counterex_1}
    \bigl\lvert \log \bigl(\frac{\nu_t}{\mu_t}\bigr)'\bigr(x)\rvert^2 
    \ge \frac{R^2}{(1+t)^2}\,(1-o(1))
    \quad \text{for } x \in B_R,
\end{align}
 as well as
\begin{align}\label{eq:counterex_2}
    \nu_t(B_R) \ge p_R\,(1-o(1))\,.
\end{align}
Note that the proof of \eqref{eq:counterex_2} is straightforward, as by definition of $\nu_t$, we have that
\begin{align*}
    \nu_t(B_R) &\geq p_R\, (\delta_R * N(0,t))(B_R) \\
    &= p_R\, \PP(|Y_t| \leq R^{1/3}) \\
    &\geq p_R\,(1 - 2\exp(-R^{2/3}/2t))\\
    &= p_R\,(1 - o(1))\,,
\end{align*}
where $Y_t \sim \cN(0,t)$, and the second inequality is a standard Gaussian tail bound. The proof of \eqref{eq:counterex_1} exploits linearity of the heat semigroup and tedious estimates; see Appendix~\ref{app:fi_lb_proof} for the remaining details of the proof.

\begin{remark}
Proposition~\ref{prop:fi_lb} was written for the heat semigroup. By rescaling this estimate, it is straightforward to obtain a corresponding bound for the OU process:
\begin{align*}
    \sup_{\nu\ne \mu} \frac{\FI(q_t^\nu\mmid q_t^\mu)}{W_2^2(\mu,\nu)}
    \ge \frac{{e^{2t}}}{(e^{2t}-1)^2}\,.
\end{align*}    
\end{remark}
On the other hand, the upper bound that we use in the proof of Theorem~\ref{thm:wass_stab} reads
\begin{align*}
    \sup_{\nu\ne \mu} \frac{\FI(q_t^\nu\mmid q_t^\mu)}{W_2^2(\mu,\nu)}
    &\le \frac{\exp(4\mathsf C_\mu)}{(e^t-1)^2}\,.
\end{align*}
Hence, our bound has the correct behavior in $t$, both for $t\searrow 0$ and for $t\to\infty$.

\bibliography{ref}

@article{cazelles2026statistical,
  title={Statistical Estimation of {M}onge Transport Maps via {B}renier Potentials},
  author={Cazelles, Elsa and Pauwels, Edouard and Portales, L{\'e}o},
  journal={arXiv preprint 2604.22366},
  year={2026}
}

@incollection {Fol85,
    AUTHOR = {F\"{o}llmer, Hans},
     TITLE = {An entropy approach to the time reversal of diffusion
              processes},
 BOOKTITLE = {Stochastic differential systems ({M}arseille-{L}uminy, 1984)},
    SERIES = {Lect. Notes Control Inf. Sci.},
    VOLUME = {69},
     PAGES = {156--163},
 PUBLISHER = {Springer, Berlin},
      YEAR = {1985},
}

@article{khudiakova2025optimal,
  title={{$L^\infty$}-optimal transport of anisotropic log-concave measures and exponential convergence in {F}isher’s infinitesimal model},
  author={Khudiakova, Ksenia A. and Maas, Jan and Pedrotti, Francesco},
  journal={The Annals of Applied Probability},
  volume={35},
  number={3},
  pages={1913--1940},
  year={2025},
  publisher={Institute of Mathematical Statistics}
}

@article{gozlan2025global,
  title={Global regularity estimates for optimal transport via entropic regularisation},
  author={Gozlan, Nathael and Sylvestre, Maxime},
  journal={arXiv preprint 2501.11382},
  year={2025}
}

@article{caffarelli2000monotonicity,
  title={Monotonicity properties of optimal transportation and the {FKG} and related inequalities},
  author={Caffarelli, Luis A},
  journal={Communications in Mathematical Physics},
  volume={214},
  number={3},
  pages={547--563},
  year={2000},
  publisher={Springer}
}

@article{merigot2026sharp,
    author = {M\'{e}rigot, Quentin},
    title = {Sharp stability of {B}renier maps via quantitative regularity of potentials},
    journal = {HAL-
05616391},
    year = {2026}
}

@book{peyre2019computational,
  title={Computational optimal transport: With applications to data science},
  author={Peyr{\'e}, Gabriel and Cuturi, Marco},
  year={2019},
  publisher={Now Foundations and Trends}
}

@article{chewi2025stability,
  title={Stability of the {K}im--{M}ilman flow map},
  author={Chewi, Sinho and Pooladian, Aram-Alexandre and Zhang, Matthew S.},
  journal={arXiv preprint 2511.01154},
  year={2025}
}

@article {brenier1991polar,
    AUTHOR = {Brenier, Yann},
     TITLE = {Polar factorization and monotone rearrangement of
              vector-valued functions},
   JOURNAL = {Comm. Pure Appl. Math.},
  FJOURNAL = {Communications on Pure and Applied Mathematics},
    VOLUME = {44},
      YEAR = {1991},
    NUMBER = {4},
     PAGES = {375--417},
}

@article{delalande2023quantitative,
  title={Quantitative stability of optimal transport maps under variations of the target measure},
  author={Delalande, Alex and M\'erigot, Quentin},
  journal={Duke Mathematical Journal},
  volume={172},
  number={17},
  pages={3321--3357},
  year={2023},
  publisher={Duke University Press}
}

@article{conforti2025coupling,
  title={A coupling approach to {L}ipschitz transport maps},
  author={Conforti, Giovanni and Eichinger, Katharina},
  journal={arXiv preprint 2502.01353},
  year={2025}
}

@article{serres2026contractive,
  title={Contractive transport maps from {$\mathbb S^2$} to nearly spherical surfaces with positive {R}icci curvature},
  author={Serres, Jordan},
  journal={Nonlinear Analysis},
  volume={267},
  pages={114058},
  year={2026},
  publisher={Elsevier}
}

@article{lopez2025bakry,
  title={A {B}akry--{E}mery approach to {L}ipschitz transportation on manifolds},
  author={L{\'o}pez-Rivera, Pablo},
  journal={Potential Analysis},
  volume={62},
  number={2},
  pages={331--353},
  year={2025},
  publisher={Springer}
}

@article{neeman2022lipschitz,
  title={Lipschitz changes of variables via heat flow},
  author={Neeman, Joe},
  journal={arXiv preprint 2201.03403},
  year={2022}
}

@inproceedings{klartag2023spectral,
  title={Spectral monotonicity under {G}aussian convolution},
  author={Klartag, Bo’az and Putterman, Eli},
  booktitle={Annales de la Facult{\'e} des sciences de Toulouse: Math{\'e}matiques},
  volume={32},
  number={5},
  pages={939--967},
  year={2023}
}

@article{fathi2024transportation,
  title={Transportation onto log-{L}ipschitz perturbations},
  author={Fathi, Max and Mikulincer, Dan and Shenfeld, Yair},
  journal={Calculus of Variations and Partial Differential Equations},
  volume={63},
  number={3},
  pages={61},
  year={2024},
  publisher={Springer}
}

@article{otto2000generalization,
  title={Generalization of an inequality by {T}alagrand and links with the logarithmic {S}obolev inequality},
  author={Otto, Felix and Villani, C{\'e}dric},
  journal={Journal of Functional Analysis},
  volume={173},
  number={2},
  pages={361--400},
  year={2000},
  publisher={Elsevier}
}

@article {eberle2016reflection,
    AUTHOR = {Eberle, Andreas},
     TITLE = {Reflection couplings and contraction rates for diffusions},
   JOURNAL = {Probab. Theory Related Fields},
  FJOURNAL = {Probability Theory and Related Fields},
    VOLUME = {166},
      YEAR = {2016},
    NUMBER = {3-4},
     PAGES = {851--886},
}

@incollection {MikShe23HeatFlow,
    AUTHOR = {Mikulincer, Dan and Shenfeld, Yair},
     TITLE = {On the {L}ipschitz properties of transportation along heat flows},
 BOOKTITLE = {Geometric aspects of functional analysis},
    SERIES = {Lecture Notes in Math.},
    VOLUME = {2327},
     PAGES = {269--290},
 PUBLISHER = {Springer, Cham},
      YEAR = {2023},
}

@article{kitagawa2025stability,
  title={Stability of optimal transport maps on {R}iemannian manifolds},
  author={Kitagawa, Jun and Letrouit, Cyril and M{\'e}rigot, Quentin},
  journal={arXiv preprint 2504.05412},
  year={2025}
}

@article{carlier2024displacement,
  title={Displacement smoothness of entropic optimal transport},
  author={Carlier, Guillaume and Chizat, L{\'e}na{\"\i}c and Laborde, Maxime},
  journal={ESAIM: Control, Optimisation and Calculus of Variations},
  volume={30},
  pages={25},
  year={2024},
  publisher={EDP Sciences}
}

@article{letrouit2024gluing,
  title={Gluing methods for quantitative stability of optimal transport maps},
  author={Letrouit, Cyril and M{\'e}rigot, Quentin},
  journal={arXiv preprint 2411.04908},
  year={2024}
}

@article{balakrishnan2025stability,
  title={Stability bounds for smooth optimal transport maps and their statistical implications},
  author={Balakrishnan, Sivaraman and Manole, Tudor},
  journal={arXiv preprint 2502.12326},
  year={2025}
}

@InProceedings{wibisono2025mixing,
  title = 	 {Mixing time of the proximal sampler in relative {F}isher information via strong data processing inequality (extended abstract)},
  author =       {Wibisono, Andre},
  booktitle = 	 {Proceedings of Thirty Eighth Conference on Learning Theory},
  pages = 	 {5716--5717},
  year = 	 {2025},
  editor = 	 {Haghtalab, Nika and Moitra, Ankur},
  volume = 	 {291},
  series = 	 {Proceedings of Machine Learning Research},
  month = 	 {7},
  publisher =    {PMLR},
}

@article{divol2025tight,
  title={Tight stability bounds for entropic {B}renier maps},
  author={Divol, Vincent and Niles-Weed, Jonathan and Pooladian, Aram-Alexandre},
  journal={International Mathematics Research Notices},
  volume={2025},
  number={7},
  pages={rnaf078},
  year={2025},
  publisher={Oxford University Press}
}

@article {BriPed25HeatFlow,
    AUTHOR = {Brigati, Giovanni and Pedrotti, Francesco},
     TITLE = {Heat flow, log-concavity, and {L}ipschitz transport maps},
   JOURNAL = {Electron. Commun. Probab.},
  FJOURNAL = {Electronic Communications in Probability},
    VOLUME = {30},
      YEAR = {2025},
     PAGES = {1--12},
publisher = {Institute of Mathematical Statistics and Bernoulli Society},
}

@article{kim2012generalization,
  title={A generalization of {C}affarelli’s contraction theorem via (reverse) heat flow},
  author={Kim, Young-Heon and Milman, Emanuel},
  journal={Mathematische Annalen},
  volume={354},
  number={3},
  pages={827--862},
  year={2012},
  publisher={Springer}
}

@article{conforti2025projected,
  title={Projected {L}angevin dynamics and a gradient flow for entropic optimal transport},
  author={Conforti, Giovanni and Lacker, Daniel and Pal, Soumik},
  journal={Journal of the European Mathematical Society},
  year={2025}
}

@article{rigollet2025sample,
  title={On the sample complexity of entropic optimal transport},
  author={Rigollet, Philippe and Stromme, Austin J.},
  journal={The Annals of Statistics},
  volume={53},
  number={1},
  pages={61--90},
  year={2025},
  publisher={Institute of Mathematical Statistics}
}

@article{letrouit2025quantitative,
  title={Quantitative stability of optimal transport},
  author={Letrouit, Cyril},
  journal={Notes du cours Peccot},
  volume={2025},
  year={2025}
}

@inproceedings{pooladian2023minimax,
  title={Minimax estimation of discontinuous optimal transport maps: the semi-discrete case},
  author={Pooladian, Aram-Alexandre and Divol, Vincent and Niles-Weed, Jonathan},
  booktitle={International Conference on Machine Learning},
  pages={28128--28150},
  year={2023},
  organization={PMLR}
}

@article{ge2025generalization,
  title={A generalization of {C}affarelli's contraction theorem to nearly spherical manifolds},
  author={Ge, Yuxin and Serres, Jordan},
  journal={arXiv preprint 2512.01496},
  year={2025}
}

@article{pooladian2021entropic,
  title={Entropic estimation of optimal transport maps},
  author={Pooladian, Aram-Alexandre and Niles-Weed, Jonathan},
  journal={arXiv preprint 2109.12004},
  year={2021}
}

@article{conforti2024weak,
  title={Weak semiconvexity estimates for {S}chr{\"o}dinger potentials and logarithmic {S}obolev inequality for {S}chr{\"o}dinger bridges},
  author={Conforti, Giovanni},
  journal={Probability Theory and Related Fields},
  volume={189},
  number={3},
  pages={1045--1071},
  year={2024},
  publisher={Springer}
}

@article{chewi2023entropic,
  title={An entropic generalization of {C}affarelli’s contraction theorem via covariance inequalities},
  author={Chewi, Sinho and Pooladian, Aram-Alexandre},
  journal={Comptes Rendus. Math{\'e}matique},
  volume={361},
  number={G9},
  pages={1471--1482},
  year={2023}
}

@article{manole2024plugin,
  title={Plugin estimation of smooth optimal transport maps},
  author={Manole, Tudor and Balakrishnan, Sivaraman and Niles-Weed, Jonathan and Wasserman, Larry},
  journal={The Annals of Statistics},
  volume={52},
  number={3},
  pages={966--998},
  year={2024},
  publisher={Institute of Mathematical Statistics}
}

@article {Leo14Schrodinger,
    AUTHOR = {L\'{e}onard, Christian},
     TITLE = {A survey of the {S}chr\"{o}dinger problem and some of its
              connections with optimal transport},
   JOURNAL = {Discrete Contin. Dyn. Syst.},
  FJOURNAL = {Discrete and Continuous Dynamical Systems. Series A},
    VOLUME = {34},
      YEAR = {2014},
    NUMBER = {4},
     PAGES = {1533--1574},
}

@article{gigli2011holder,
  title={On {H}{\"o}lder continuity-in-time of the optimal transport map towards measures along a curve},
  author={Gigli, Nicola},
  journal={Proceedings of the Edinburgh Mathematical Society},
  volume={54},
  number={2},
  pages={401--409},
  year={2011},
  publisher={Cambridge University Press}
}

@article{altschuler_shifted_2025,
	title = {Shifted composition {I}: {H}arnack and reverse transport inequalities},
	volume = {71},
	shorttitle = {Shifted {Composition} {I}},
	number = {1},
	urldate = {2026-05-11},
	journal = {IEEE Transactions on Information Theory},
	author = {Altschuler, Jason M. and Chewi, Sinho},
	month = jan,
	year = {2025},
	pages = {90--113},
}

@article{chaintron_propagation_2025,
	title = {Propagation of weak log-concavity along generalised heat flows via {H}amilton--{J}acobi equations},
	journal = {arXiv preprint 2508.07931},
	author = {Chaintron, Louis-Pierre and Conforti, Giovanni and Eichinger, Katharina},
	month = aug,
	year = {2025},
}

@article {wang1997logarithmic,
    AUTHOR = {Wang, Feng-Yu},
     TITLE = {Logarithmic {S}obolev inequalities on noncompact {R}iemannian manifolds},
   JOURNAL = {Probab. Theory Related Fields},
  FJOURNAL = {Probability Theory and Related Fields},
    VOLUME = {109},
      YEAR = {1997},
    NUMBER = {3},
     PAGES = {417--424},
}

@article {FatGozPro20Caff,
    AUTHOR = {Fathi, Max and Gozlan, Nathael and Prod'homme, Maxime},
     TITLE = {A proof of the {C}affarelli contraction theorem via entropic regularization},
   JOURNAL = {Calc. Var. Partial Differential Equations},
  FJOURNAL = {Calculus of Variations and Partial Differential Equations},
    VOLUME = {59},
      YEAR = {2020},
    NUMBER = {3},
     PAGES = {Paper No. 96, 18},
}

@article {BobGenLed01Hyper,
    AUTHOR = {Bobkov, Sergey G. and Gentil, Ivan and Ledoux, Michel},
     TITLE = {Hypercontractivity of {H}amilton--{J}acobi equations},
   JOURNAL = {J. Math. Pures Appl. (9)},
  FJOURNAL = {Journal de Math\'ematiques Pures et Appliqu\'ees. Neuvi\`eme
              S\'erie},
    VOLUME = {80},
      YEAR = {2001},
    NUMBER = {7},
     PAGES = {669--696},
}

@article {PolWu16WassCont,
    AUTHOR = {Polyanskiy, Yury and Wu, Yihong},
     TITLE = {Wasserstein continuity of entropy and outer bounds for interference channels},
   JOURNAL = {IEEE Trans. Inform. Theory},
  FJOURNAL = {Institute of Electrical and Electronics Engineers.
              Transactions on Information Theory},
    VOLUME = {62},
      YEAR = {2016},
    NUMBER = {7},
     PAGES = {3992--4002},
}

@InProceedings{MerDelCha20Stability,
  title = 	 {Quantitative stability of optimal transport maps and linearization of the 2-Wasserstein space},
  author =       {M\'erigot, Quentin and Delalande, Alex and Chazal, Frederic},
  booktitle = 	 {Proceedings of the Twenty Third International Conference on Artificial Intelligence and Statistics},
  pages = 	 {3186--3196},
  year = 	 {2020},
  editor = 	 {Chiappa, Silvia and Calandra, Roberto},
  volume = 	 {108},
  series = 	 {Proceedings of Machine Learning Research},
  month = 	 {aug},
  publisher =    {PMLR},
}

@article{letrouit_unstable_2026,
	title = {Unstable optimal transport maps},
	volume = {364},
	language = {en},
	number = {G2},
	journal = {Comptes Rendus. Mathématique},
	author = {Letrouit, Cyril},
	month = may,
	year = {2026},
	pages = {333--344},
}

@inproceedings{silveri2025beyond,
  title={Beyond log-concavity and score regularity: Improved convergence bounds for score-based generative models in W2-distance},
  author={Silveri, Marta Gentiloni and Ocello, Antonio},
  booktitle={Forty-second International Conference on Machine Learning},
  year={2025}
}

\newpage
\appendix

\section{Existence of Kim--Milman flow maps}\label{app:existence}

\begin{definition}\label{def:reg_approx}
Given $\nu \in \cP_2(\R^d)$, we say that $\{\nu_n\}_{n\in\N}$ is a regular approximating sequence for $\nu$ if each $\nu_n$ satisfies~\eqref{theta} and $W_2(\nu_n,\nu) \to 0$.
\end{definition}

For example, one can take
\begin{align*}
    \nu_n = \frac{\nu|_{B(0,n)}}{\nu(B(0,n))} * \cN(0,n^{-1} I)\,.
\end{align*}

{We need a few preliminary lemmas.}

\begin{lemma}\label{lem:technical_estimates}
    Let $\nu$ have finite second moment: $\mom(\nu) \le M$. For all $t\ge \delta > 0$ and all $y$ with $\norm y \le R$,
        \begin{align*}
            \bigl\lVert \nabla \log \frac{q_t^\nu}{\gamma}(y)\bigr\rVert \le C(\delta,R,M)\,e^{-t} \qquad\text{and}\qquad
            \bigl\lVert \nabla^2 \log \frac{q_t^\nu}{\gamma}(y)\bigr\rVert_{\rm op} \le C'(\delta,R,M)\,e^{-2t}\,.
        \end{align*}
\end{lemma}
\begin{proof}
    For the first statement, by~\eqref{eq:tweedie_one},
    \begin{align*}
        \nabla \log \frac{q_t^\nu}{\gamma}(y)
        &= -\frac{e^{-t}}{1-e^{-2t}}\, \E[e^{-t} X_t - X_0 \mid X_t = y]\,.
    \end{align*}
    So,
    \begin{align*}
        \Bigl\lVert \nabla \log \frac{q_t^\nu}{\gamma}(y)\Bigr\rVert
        &\le \frac{e^{-t}}{1-e^{-2t}}\,\bigl(e^{-t}\,\|y\| + \E[\|X_0\| \mid X_t=y]\bigr)\,.
    \end{align*}
    Also,
    \begin{align*}
        \E[\|X_0\| \mid X_t=y]
        &\le \int \|x\|\exp\Bigl(-\frac{\|y-e^{-t}x\|^2}{2(1-e^{-{2}t})}\Bigr)\,\nu(\!\dd x) \mathop{\Big/} \int \exp\Bigl(-\frac{\|y-e^{-t} x\|^2}{2(1-e^{-{2}t})}\Bigr)\,\nu(\!\dd x)\,.
    \end{align*}
    The numerator is bounded by $M$, whereas the denominator is lower bounded by
    \begin{align*}
        \int \exp\Bigl(-\frac{\|y-e^{-t} x\|^2}{2(1-e^{-{2}t})}\Bigr)\,\nu(\!\dd x)
        &\ge c_* \int_{B(0,2M)} \nu(\!\dd x)
        \ge \frac{3c_*}{4}
    \end{align*}
    by Markov's inequality, where
    \begin{align*}
        c_* \deq \inf\Bigl\{\exp\Bigl(-\frac{\|y-e^{-t} x\|^2}{2(1-e^{-{2}t})}\Bigr) \Bigm\vert t \ge \delta,\, \|x\| \le 2M,\, \|y\| \le R \Bigr\} > 0\,.
    \end{align*}

    The second statement is similar, using~\eqref{eq:tweedie_two} instead.
\end{proof}

\begin{lemma}\label{lem:km_regular_case}
    Let $\nu$ satisfy~\eqref{theta}.
    Then, the convergence $\KM^{\nu,T}_{T-t}(x) \to \KM^\nu_{t\gets}(x)$ holds uniformly over $(t,x)$ belonging to compact subsets of $(0,\infty)\times \R^d$.
    Moreover, for all $t > 0$,
    \begin{align}\label{eq:deriv_km_map}
        \partial_t \KM^\nu_{t\gets} = -\nabla \log \frac{q_t^\nu}{\gamma}\circ \KM^\nu_{t\gets}\,,
    \end{align}
    and $(\KM^\nu_{t\gets})_\sharp \gamma = q_t^\nu$.
\end{lemma}
\begin{proof}
    Consider the mapping $F^T : (t,x) \mapsto \KM^{\nu,T}_{T-t}(x)$.
    By definition,
    \begin{align*}
        \partial_t F^T(t,x)
        &= -\partial_s\big|_{s=T-t} \KM^{\nu,T}_s(x)
        = -\nabla \log \frac{q^\nu_{T-s}}{\gamma}(\KM^{\nu,T}_s(x))\Big|_{s=T-t}
        = -\nabla \log \frac{q^\nu_t}{\gamma}(F^T(t,x))\,.
    \end{align*}
    Hence, since $F^T(T,x) = x$,
    \begin{align}\label{eq:F_int_eq}
        F^T(t,x) = x + \int_t^T \nabla \log \frac{q_s^\nu}{\gamma}\circ F^T(s,x)\dd s\,.
    \end{align}
    Thus, for $T' > T$,
    \begin{align*}
        F^{T'}(t,x) - F^T(t,x)
        &= \int_t^T \bigl[\nabla \log \frac{q_s^\nu}{\gamma}(F^{T'}(s,x))-\nabla \log \frac{q_s^\nu}{\gamma}(F^{T}(s,x))\bigr]\dd s \\
        &\qquad{} + \int_T^{T'} \nabla \log\frac{q_s^\nu}{\gamma}(F^{T'}(s,x))\dd s\,.
    \end{align*}
    Under~\eqref{theta}, we know by Lemma~\ref{lem:technical_estimates} that $\nabla \log (q_s^\nu/\gamma)$ is $\ell_s$-Lipschitz, where $\int_t^\infty \ell_s\dd s < \infty$ for all $t > 0$.
    Moreover, again by Lemma~\ref{lem:technical_estimates}, with $M \deq \mom(\nu)$ and for all $\|x\|\le R$,
    \begin{align*}
        \partial_t \|F^{T'}(t,x)\|
        &\le \bigl\lVert \nabla \log \frac{q_t^\nu}{\gamma}(F^{T'}(t,x))\bigr\rVert
        \le \ell_t\,\|F^{T'}(t,x)\| + \bigl\lVert \nabla \log \frac{q_t^\nu}{\gamma}(0)\bigr\rVert \\
        &\le \ell_t\,\|F^{T'}(t,x)\| + C(t,M,1)\,.
    \end{align*}
    By Gr\"onwall's inequality, $\|F^{T'}(s,x)\| \le R' \deq R'(t,M,R)$ for all $s\ge t$.
    Hence,
    \begin{align*}
        \|F^{T'}(t,x) - F^T(t,x)\|
        &\le \int_t^T \ell_s\,\|F^{T'}(s,x) - F^T(s,x)\|\dd s + \int_T^{T'} C(t,M,R')\,e^{-s}\dd s\,.
    \end{align*}
    A standard application of Gr\"onwall's inequality now shows that $\{F^T\}_{T> 0}$ is Cauchy with respect to uniform convergence over compact subsets of $(0,\infty) \times \R^d$.
    Hence, the limiting map $\KM^\nu_{t\gets}$ is well defined for all $t > 0$.

    Passing to the limit in~\eqref{eq:F_int_eq}, we deduce that
    \begin{align*}
        \KM^\nu_{t\gets}(x) = x + \int_t^\infty \nabla \log \frac{q_s^\nu}{\gamma} \circ \KM^\nu_{s\gets}(x)\dd s\,.
    \end{align*}
    By the fundamental theorem of calculus, we conclude that~\eqref{eq:deriv_km_map} holds.

    For the last statement, let $X_0 \sim \nu$ and $Z \sim \gamma$ be independent, and for $t \ge 0$ set $X_t \deq e^{-t}\,X_0 + \sqrt{1-e^{-2t}}\,Z$.
    Then, $\KM^{\nu,T}_{T-t}(X_T) \sim q_t^\nu$ for all $T > 0$.
    By the uniform convergence on compacts, and since $X_T \to Z$ almost surely, we have $\KM^\nu_{t \gets}(Z) \sim q_t^\nu$.
\end{proof}

\begin{proof}[Proof of Lemma~\ref{lem:existence_km}]
    \textbf{We first show that $\{\KM^{\nu_n}_{\delta\gets}\}_{n\in\N}$ is Cauchy in $L^2(\gamma)$ for each fixed $\delta > 0$.}
    By Lemma~\ref{lem:km_regular_case}, we have
    \begin{align*}
        \|\KM^{\nu_m}_{\delta\gets} - \mathsf T^{\nu_{n}}_{\delta\gets}\|
        &\le \int_\delta^\infty\Bigl\lVert \nabla \log \frac{q_t^{\nu_m}}{\gamma} \circ \KM^{\nu_m}_{\gets t} - \nabla \log \frac{q_t^{\nu_n}}{\gamma} \circ \KM^{\nu_n}_{\gets t}\Bigr\rVert\dd t\,.
    \end{align*}
    Since $W_2(\nu_n,\nu) \to 0$ and $\nu \in \cP_2(\R^d)$, we have that $\sup_{n\in\N} \mom(\nu_n) \le M < \infty$.

    Let $E_R$ be the set on which $\|\mathsf T_{t\gets}^{\nu_m}\| \vee \|\mathsf T_{t\gets}^{\nu_n}\| \le R$ for all $t \ge \delta$.
    Then, on $E_R$, by Lemma~\ref{lem:technical_estimates},
    \begin{align*}
        &\Bigl\lVert \nabla \log \frac{q_t^{\nu_m}}{\gamma} \circ \mathsf T^{\nu_m}_{t\gets} - \nabla \log \frac{q_t^{\nu_n}}{\gamma} \circ \mathsf T^{\nu_n}_{t\gets}\Bigr\rVert \\
        &\qquad \le C'(\delta,R,M)\,e^{-2t}\,\|\mathsf T_{t\gets}^{\nu_m} - \mathsf T_{t\gets}^{\nu_n}\| + \sup_{B(0,R)}{\Bigl\lVert\nabla \log \frac{q_t^{\nu_m}}{\gamma} - \nabla \log \frac{q_t^{\nu_n}}{\gamma} \Bigr\rVert}\,.
    \end{align*}
    By Gr\"{o}nwall's inequality, it follows that
    \begin{align*}
        \|\mathsf T^{\nu_{m}}_{\delta\gets} - \mathsf T^{\nu_{n}}_{\delta\gets}\|
        &\lesssim_{\delta,R} \int_\delta^\infty\sup_{B(0,R)}{\Bigl\lVert\nabla \log \frac{q_t^{\nu_m}}{\gamma} - \nabla \log \frac{q_t^{\nu_n}}{\gamma} \Bigr\rVert}\dd t
        = o_{\delta,R}(1)
    \end{align*}
    where $o_{\delta,R}(1)$ denotes a term that vanishes as $m,n\to\infty$ for fixed $\delta$ and $R$.
    This can again be seen from the identity~\eqref{eq:tweedie_one}; we have
    \begin{align*}
        &\nabla \log \frac{q_t^{\nu_m}}{\gamma}(y) - \nabla \log \frac{q_t^{\nu_n}}{\gamma}(y) \\
        &\qquad = \frac{e^{-t}}{1-e^{-2t}}\,\biggl[\frac{\int x \exp(-\frac{\|y-e^{-t} x\|^2}{2(1-e^{-2t})})\,\nu_m(\!\dd x)}{\int \exp(-\frac{\|y-e^{-t} x\|^2}{2(1-e^{-2t})})\,\nu_m(\!\dd x)} - \frac{\int x \exp(-\frac{\|y-e^{-t} x\|^2}{2(1-e^{-2t})})\,\nu_n(\!\dd x)}{\int \exp(-\frac{\|y-e^{-t} x\|^2}{2(1-e^{-2t})})\,\nu_n(\!\dd x)}\biggr]\,.
    \end{align*}
    For $\delta \le t \le T$ and $\norm y \le R$, all of the integrands are uniformly bounded and Lipschitz, so for each $t$ we have $\sup_{B(0,R)}{\bigl\lVert\nabla \log \frac{q_t^{\nu_m}}{\gamma} - \nabla \log \frac{q_t^{\nu_n}}{\gamma} \bigr\rVert} \to 0$.
    By Lemma~\ref{lem:technical_estimates}, we can apply the dominated convergence theorem to deduce that the integral over $[\delta,T]$ converges to zero as $m,n\to\infty$.
    On the other hand, Lemma~\ref{lem:technical_estimates} shows that the integral over $[T,\infty)$ vanishes as $T\to\infty$, uniformly over $m,n\in\N$.

    Next, by~\eqref{eq:deriv_km_map}, we can estimate:
    \begin{align*}
        \sup_{n\in\N} \gamma(E_R^\comp)
        &\le 2\sup_{n\in\N} \frac{\|\sup_{\delta\le t <\infty}{\|\KM_{t\gets}^{\nu_n}}\|\,\|_{L^2(\gamma)}^2}{R^2}\\[0.25em]
        &\lesssim \frac{1}{R^2} \sup_{n\in\N}{\Bigl\{d + \Bigl(\int_\delta^\infty \sqrt{\FI(q_t^{\nu_n}\mmid \gamma)}\dd t \Bigr)^2 \Bigr\}}
        = O\Bigl(\frac{1}{R^2}\Bigr)\,.
    \end{align*}
    By letting $m,n\to\infty$ first and then $R \to \infty$, we deduce that
    \begin{align*}
        \|\mathsf T^{\nu_{m}}_{\delta\gets} - \mathsf T^{\nu_{n}}_{\delta \gets}\|\to 0 \qquad\text{in probability}\,.
    \end{align*}
    On the other hand, $W_2(\nu_n,\nu) \to 0$ implies that $\{\mathsf \|\KM^{\nu_n}_{\delta\gets}\|^2\}_{n\in\N}$ is uniformly integrable, since $(\KM^{\nu_n}_{\delta\gets})_\sharp \gamma = q_\delta^{\nu_n}$ and $W_2(q_\delta^{\nu_n}, q_\delta^\nu) \to 0$. Thus,
    \begin{align*}
        \|\mathsf T^{\nu_m}_{\delta\gets} - \mathsf T^{\nu_{n}}_{\delta\gets}\|_{L^2(\gamma)}\to 0\,.
    \end{align*}
    Hence, the $L^2(\gamma)$ limit $\KM^\nu_{\delta\gets} \deq \lim_{n\to\infty} \KM^{\nu_n}_{\delta\gets}$ exists.
    Moreover, by inspecting the above proof, one can readily see that the limit is independent of the regular approximating sequence, and for any bounded continuous test function $\phi$,
    \begin{align*}
        \int \phi\circ \KM^\nu_{\delta\gets}\dd \gamma = \lim_{n\to\infty} \int \phi\circ\KM^{\nu_n}_{\delta\gets}\dd\gamma
        = \lim_{n\to\infty} \int \phi\dd q_\delta^{\nu_n} = \int \phi\dd q_\delta^\nu\,,
    \end{align*}
    so that $(\KM^\nu_{\delta\gets})_\sharp \gamma = q_\delta^\nu$.

    \textbf{Next, we show that $\{\KM^\nu_{\delta\gets}\}_{\delta > 0}$ is Cauchy in $L^2(\gamma)$ as $\delta \searrow 0$.}
    By~\eqref{eq:deriv_km_map}, for $0 < s \le t$,
    \begin{align*}
        \|\mathsf T^{\nu_n}_{s\gets} - \mathsf T^{\nu_n}_{t\gets}\|_{L^2(\gamma)}
        &\le \int_s^t \Bigl\lVert\nabla \log \frac{q_u^{\nu_n}}{\gamma} \circ \KM_{u\gets}^{\nu_n}\Bigr\rVert_{L^2(\gamma)}\dd u
        = \int_s^t \sqrt{\FI(q_u^{\nu_n} \mmid \gamma)}\dd u\,.
    \end{align*}
    By Lemma~\ref{lem:FI_abs_bound} and by passing $n\to\infty$, it follows that
    \begin{align*}
        \|\mathsf T^{\nu}_{s\gets} - \mathsf T^{\nu}_{t\gets}\|_{L^2(\gamma)}
        \le O(\sqrt t)\,.
    \end{align*}
    This proves the claim and shows that we can define the $L^2(\gamma)$ limit $\KM^\nu \deq \lim_{\delta\searrow 0} \KM^\nu_{\delta\gets}$.
    By the same argument as before, $(\KM^\nu)_\sharp \gamma = \nu$ as desired.
\end{proof}

\section{Remaining proofs}
\subsection{A technical lemma}
\begin{lemma}\label{lem:FI_abs_bound}
    For any probability measure $\mu$ with $\mom(\mu) \defeq (\int \|x\|^2\,\mu(\!\dd x))^{1/2} < \infty$ and any $t > 0$,
    \begin{align*}
        \FI(q_t^\mu\mmid \gamma)
        &\le e^{-2t}\,\mom^2(\mu) + \frac{e^{-4t}}{1-e^{-2t}}\,d\,.
    \end{align*}
    In particular, if $\mom(\mu) \le \sqrt d$,
    \begin{align*}
        \FI(q_t^\mu\mmid \gamma)
        &\le \frac{de^{-2t}}{1-e^{-2t}}\,.
    \end{align*}
\end{lemma}
\begin{proof}
    Let $X_0 \sim \mu$ and $Z \sim \gamma$ be independent, and $X_t = e^{-t}\,X_0 + \sqrt{1-e^{-2t}}\,Z$.
    By Tweedie's identity~\eqref{eq:tweedie_one}
    and Jensen's inequality,
    \begin{align*}
        \FI(q_t^\mu\mmid\gamma)
        &\le \frac{e^{-2t}}{(1-e^{-2t})^2}\, \E[\|e^{-t}\,X_t - X_0 \|^2] \\
        &= \frac{e^{-2t}}{(1-e^{-2t})^2}\, \E[\|e^{-t}\sqrt{1-e^{-2t}}\,Z - (1-e^{-2t})\, X_0 \|^2] \\
        &= \frac{e^{-2t}}{(1-e^{-2t})^2}\,\{(1-e^{-2t})^2\,\mom^2(\mu) + e^{-2t}\,(1-e^{-2t})\,d\} \\
        &= e^{-2t}\,\mom^2(\mu) + \frac{de^{-4t}}{1-e^{-2t}}\,. \qedhere
    \end{align*}
\end{proof}

\subsection{Remaining proofs from Proposition~\ref{prop:fi_lb}}\label{app:fi_lb_proof}

\subsubsection{Proof of \texorpdfstring{\eqref{eq:counterex_0}}{the first estimate}}

For $x \in A_R = [R/(1+t), R/(1+t) + R^{-1}]$, we compute
\begin{align*}
    R - x = R - \frac{R}{(1+t)} + O(R^{-1}) = \frac{tR}{(1+t)} + O(R^{-1}) = \frac{tR}{(1+t)}\,(1 + O(R{^{-2}}))\,,
\end{align*}
where we absorb factors of $t$ into the asymptotic notation. From here it is clear that
\begin{align*}
    W_2^2(\mu,\nu) \leq \int_{A_R}|x-R|^2 \,\mu(\!\dd x) = \frac{t^2R^2}{(1+t)^2}\,(1+o(1))\,\mu(A_R) = \frac{p_R t^2 R^2}{(1+t)^2}\,(1+o(1))\,.
\end{align*}

\subsubsection{Proof of \texorpdfstring{\eqref{eq:counterex_1}}{the second estimate}}

By linearity of the heat semigroup, we have that
\begin{align*}
    \nu_t = \mu_t - \alpha + p_R \beta\,,
\end{align*}
where $\beta \deq \cN(R,t)$ and
\begin{align*}
    \alpha(x) = \int_{A_R} \frac{\exp(-\frac{1}{2t}\,|y-x|^2)}{\sqrt{2\pi t}}\, \mu(\!\dd y)\,.
\end{align*}
Our goal is to obtain a lower bound on 
\begin{align*}
    \Bigl|\log\Bigl(\frac{\nu_t}{\mu_t}\Bigr)'(x)\Bigr|
\end{align*}
which is uniform in $x \in B_R$. Note that
\begin{align*}
    \log\Bigl(\frac{\nu_t}{\mu_t}\Bigr)'(x) = \frac{\nu_t'}{\nu_t}(x) - \frac{\mu_t'}{\mu_t}(x) = \frac{\nu_t'}{\nu_t}(x) + \frac{x}{1+t}\,,
\end{align*}
where the last equality follows directly since $\mu_t = \cN(0,1+t)$. Following the expression for $\nu_t$, we can write
\begin{align*}
    \nu_t'(x) &= \mu_t'(x) - \alpha'(x) + p_R \beta'(x) \\
    &=-\frac{x}{(1+t)}\,\mu_t(x) - \alpha'(x) - \frac{(x-R)}{t}\,p_R\beta(x)\,,
\end{align*}
where 
\begin{align*}
    \alpha'(x) = -\frac{x}{t}\,\alpha(x) + \frac{1}{t}\int_{A_R}y\, (2\pi t)^{-1/2}\exp(-\tfrac{1}{2t}\,(x-y)^2)\, \mu({\rm d}y)\,.
\end{align*}
The bound then follows from the following estimates over $x \in B_R$
\begin{align}\label{eq:claim_app}
    \frac{\mu_t(x)}{p_R\beta (x)} = o(1)\,, \quad \frac{\alpha(x)}{p_R\beta (x)} = o(1)\,, \quad \frac{\nu_t(x)}{p_R\beta (x)} = 1+o(1)\,,
\end{align}
{as well as the estimates for the derivatives}
\begin{align}\label{eq:claim_app_2}
    \frac{\mu_t'(x)}{p_R \beta(x)} = o(1)\,, \qquad \frac{\alpha'(x)}{p_R\beta(x)} = o(1)\,.
\end{align}
Indeed, from \eqref{eq:claim_app} and~\eqref{eq:claim_app_2}, we can conclude that
\begin{align*}
    \frac{\nu_t'(x)}{\nu_t(x)}
    &= \frac{\nu_t'(x)}{p_R \beta(x)}\,(1+o(1))
    = \frac{\beta'(x)}{\beta(x)}\,(1+o(1))
    = \frac{\beta'(x)}{\beta(x)} + o(R^{1/3})
    = -\frac{x-R}{t} + o(R^{1/3})\,,
\end{align*}
whence
\begin{align*}
    \bigl\lvert \frac{\nu_t'(x)}{\nu_t(x)} - \frac{\mu_t'(x)}{\mu_t(x)}\bigr\rvert
    &= \bigl\lvert \frac{x}{1+t} - \frac{x-R}{t} + o(R^{1/3})\bigr\rvert
    = \frac{R}{1+t} + O(R^{1/3})
    = \frac{R}{1+t}\,(1+o(1))\,.
\end{align*}
as desired.

We now prove the estimates in \eqref{eq:claim_app} and~\eqref{eq:claim_app_2}.
First, we calculate
\begin{align}\label{eq:mut_beta}
    \log\frac{\mu_t(x)}{p_R \beta(x)}
    &= -\frac{x^2}{2(1+t)} -\frac{1}{2} \log(2\pi(1+t)) -\log p_R + \frac{|x-R|^2}{2t} + \frac{1}{2} \log(2\pi t)\,,
\end{align}
where it suffices to show that this quantity diverges to $-\infty$ as $R\to\infty$. We first compute $\log p_R$, which is
\begin{align*}
    \log p_R = \log \int_{R/(1+t)}^{R/(1+t) + R^{-1}} (2\pi)^{-1/2}\exp(-x^2/2)\dd x
    = -\frac{R^2}{2(1+t)^2} + O(\log R)\,.
\end{align*}
Substituting this back into \eqref{eq:mut_beta}, and also writing $x \in B_R$ as $x = R + s$ for $|s|\leq R^{1/3}$, we have
\begin{align*}
    \log\frac{\mu_t(x)}{p_R \beta(x)} &= -\frac{x^2}{2(1+t)} + \frac{(x-R)^2}{2t} + \frac{R^2}{2(1+t)^2} + O(\log R) \\
    &= -\frac{(R+s)^2}{2(1+t)} +\frac{R^2}{2(1+t)^2} + \frac{s^2}{2t} + O(\log R) \\
    &= {-\frac{tR^2}{2(1+t)^2} + \frac{s^2}{2t(1+t)} - \frac{Rs}{(1+t)} + O(\log R) \rightarrow -\infty}\,,
\end{align*}
where the limit holds as the dominant term scales like $-R^2$. This proves that for $x \in B_R$
\begin{align*}
\frac{\mu_t(x)}{p_R \beta(x)} = o(1)\,.
\end{align*}
{Moreover, using $\mu_t'(x) = -x\mu_t(x)/(1+t)$, it also readily implies
\begin{align*}
    \frac{\mu_t'(x)}{p_R\beta(x)} = o(1)\,.
\end{align*}
}
We now turn to the asymptotics of
\begin{align*}
    \frac{\alpha(x)}{p_R \beta(x)}
    &= \frac{1}{p_R} \int_{A_R} \exp\Bigl(-\frac{1}{2t}\,(|y-x|^2 - |x-R|^2)\Bigr) \,\mu(\!\dd y)\,.
\end{align*}
Writing $x = R+s$ for $-R^{1/3} \leq s \leq R^{1/3}$ and $y = R/(1+t) + u$ for $u \in [0,R^{-1}]$, we have the bound
\begin{align*}
    |y-x|^2 - |x-R|^2 &= \Bigl(\frac{tR}{1+t} - u\Bigr)\Bigl(\frac{tR}{1+t} - u + 2s\Bigr) \\
    &\geq \Bigl(\frac{tR}{1+t} - R^{-1}\Bigr)\Bigl(\frac{tR}{1+t} - R^{-1} - 2R^{1/3}\Bigr)\,,
\end{align*}
which is strictly positive for $R$ large enough. As a result, 
\begin{align*}
    \frac{\alpha(x)}{p_R \beta(x)} = o(1)\,,
\end{align*}
with exponentially fast decay. These two terms imply that 
\begin{align*}
    \frac{\nu_t(x)}{p_R\beta(x)} = 1 + o(1)\,.
\end{align*}
Note that this last argument can also be used to show that
\begin{align*}
    \Bigl|\frac{\alpha'(x)}{p_R\beta(x)}\Bigr| = o(1)\,.
\end{align*}
{Indeed, since
\begin{align*}
    \alpha'(x) = -\int_{A_R} \frac{x-y}{t}\,\frac{\exp(-\frac{1}{2t}\,|y-x|^2)}{\sqrt{2\pi t}}\,\mu(\!\dd y)\,,
\end{align*}
it holds that $|\alpha'(x)| \lesssim_t R\alpha(x)$ for $x\in B_R$, and the estimate follows since the ratio $\frac{\alpha(x)}{p_R\beta(x)}$ is exponentially small.
}

\subsection{An alternative proof of Theorem~\ref{thm:stab_ent}}\label{sec:proof_old}

Our first step is the following lemma.
\begin{lemma}\label{lem:ov_new}
Suppose $\mu$ satisfies \eqref{theta} with $L_t \defeq L_t(\mu)$ as in \eqref{eq:def-L_t}. For any $\nu \in \cP_2(\R^d)$,\looseness-1
\begin{align*}
\frac{\dd}{\dd t} \bigl(2 L_t^2\, {\mathsf H}(q_t^\nu\mmid q_t^\mu) \bigr)^{1/2} \leq
-\FI(q^\nu_{t}\mmid q^\mu_{t})^{1/2} - \theta_{t}\,\bigl(2 L_{t}^2\, {\mathsf H}(q_{t}^\nu\mmid q_{t}^\mu) \bigr)^{1/2} \,.
\end{align*}
\end{lemma}
\begin{proof}
Writing $\mathsf h_t \defeq 2 L_{t}^2\, {\mathsf H}(q_{t}^\nu\mmid q_{t}^\mu)$ for shorthand, we compute
\begin{align*}
\frac{\dd}{\dd t}\, \mathsf h_t &= 4 L_{t}\, {\mathsf H}(q_{t}^\nu\mmid q_t^\mu)\, \frac{\dd}{\dd t}\, L_{t} + 2L_{t}^2\, \frac{\dd}{\dd t}\, {\mathsf H}(q_{t}^\nu\mmid q_{t}^\mu) \\[0.25em]
&= -2\theta_{t}\, \mathsf h_t - 2 L_{t}^2 \FI(q_{t}^\nu\mmid q_{t}^\mu)\,,
\end{align*}
which follows from the definition of $L_t$ and de Bruijn's identity. By the chain rule, we arrive at\looseness-1
\begin{align*}
\frac{\dd}{\dd t}\, (\mathsf h_t)^{1/2} &= \frac{\frac{\dd}{\dd t}\, \mathsf h_t}{2  \mathsf h_t^{1/2}} = -\theta_{t}\,\mathsf h_t^{1/2} - \frac{L_{t} \FI(q_{t}^\nu\mmid q_{t}^\mu)}{\bigl(2{\mathsf H}(q_{t}^\nu\mmid q_{t}^\mu)\bigr)^{1/2}}\,.
\end{align*}
Using that $q_{t}^\mu = (\mathsf T^{\mu}_{t\gets})_{\sharp}\gamma$ with $\mathsf T^{\mu}_{t\gets}$ being Lipschitz continuous with constant $L_{t}(\mu)$, we deduce that
$q_{t}^\mu$ satisfies the following log-Sobolev inequality:
\begin{align*}
{\mathsf H}(q_{t}^\nu\mmid q_{t}^\mu) \leq \frac{L_{t}^2}{2}\FI(q_{t}^\nu\mmid q_{t}^\mu)\,.
\end{align*}
Altogether, we obtain
\begin{align*}
\frac{\dd}{\dd t}\, (\mathsf h_t)^{1/2} \leq -\theta_{t}\,\mathsf h_t^{1/2} - \FI(q_{t}^\nu\mmid q_{t}^\mu)^{1/2}\,,
\end{align*}
which concludes the proof.
\end{proof}

With this at hand we can now adapt the argument of \cite{otto2000generalization}.

\begin{proof}[Proof of Theorem~\ref{thm:stab_ent}]
For simplicity, we omit the approximation argument, which is the same as in the original proof of Theorem~\ref{thm:stab_ent}, and focus on the main calculation.

Starting from \eqref{eq:stab_start} and subtracting off the inequality from Lemma~\ref{lem:ov_new}, we obtain
\begin{align*}
&\frac{\dd}{\dd t} \bigl[ \|\KM^\mu_{t\gets} - \KM^\nu_{t\gets}\|_{L^2(\gamma)} - (2L^2_{t}\,{\mathsf H}(q^\nu_{t}\mmid q^\mu_{t}))^{1/2} \bigr] \\
&\qquad \qquad \geq -\theta_{t}\,\bigl[\|\KM^\mu_{t\gets} - \KM^\nu_{t\gets}\|_{L^2(\gamma)} - (2L^2_{t}\,{\mathsf H}(q^\nu_{t}\mmid q^\mu_{t}))^{1/2}\bigr]\,.
\end{align*}
A single application of Gr\"onwall's inequality then yields
\begin{align*}
    &\|\KM^\mu_{T\gets}-\KM^\nu_{T\gets}\|_{L^2(\gamma)} - \sqrt{2}L_{T}\,{\mathsf H}(q^\nu_{T} \mmid q^\mu_{T})^{1/2} \\
    &\qquad \qquad \geq \exp\Bigl(-\int_0^T \theta_t\dd t\Bigr)\,\bigl( \|\KM^\mu -\KM^\nu\|_{L^2(\gamma)} - \sqrt{2}L_{0}\,{\mathsf H}(\nu\mmid \mu)^{1/2} \bigr)\,.
\end{align*}
To conclude, we take $T \rightarrow \infty$ on both sides. By assumption, $\lim_{T\to\infty }L_T < +\infty$ while $q_T^\nu,q_T^\mu \to \gamma$, and thus the left-hand side vanishes; the resulting bound is
\begin{align*}
\|\mathsf T^\mu - \mathsf T^\nu\|_{L^2(\gamma)} - \sqrt{2}L_0(\mu)\, {\mathsf H}(\nu\mmid\mu)^{1/2} \leq 0\,.
\end{align*}
Rearranging concludes the proof.
\end{proof}

\end{document}